\documentclass[reqno,oneside,11pt]{amsart}

\calclayout

\usepackage[utf8]{inputenc}
\usepackage[T1]{fontenc}

\usepackage{amssymb, amsmath, amsthm}
\usepackage{fourier}
\usepackage{tikz-cd}
\usepackage{graphicx}
\usepackage{mathtools}
\mathtoolsset{showonlyrefs}
\usepackage{xspace}
\theoremstyle{definition}
\newtheorem{defi}{Definition}[section]
\newtheorem{ex}[defi]{Example}
\newtheorem{remark}[defi]{Remark}
\newtheorem*{convention}{Convention}
\theoremstyle{plain}
\newtheorem{theorem}[defi]{Theorem}
 \newtheorem{prop}[defi]{Proposition}
\newtheorem{lemma}[defi]{Lemma}
\newtheorem{cor}[defi]{Corollary}

\usepackage{marginnote}

\usepackage{color}

\numberwithin{equation}{section}

\usepackage[pdftitle={Continuity of eigenvalues},
pdfauthor={Alexandre Girouard, Mikhail Karpukhin, Jean Lagacé}, 
pdfborder={0 0 0}
]{hyperref}

\makeatletter

\let\TagsLeftOn\tagsleft@true
\let\TagsLeftOff\tagsleft@false
\makeatother

\newcommand{\B}{\mathbb B}
\newcommand{\D}{\mathbb{D}}
\newcommand{\R}{\mathbb R}
\newcommand{\N}{\mathbb N}
\newcommand{\Z}{\mathbb Z}

\newcommand{\de}{\, \mathrm{d}}
\newcommand{\del}{\partial}

\newcommand{\Vol}{\operatorname{Vol}}

\newcommand{\CC}{\mathcal{C}}

\newcommand{\CH}{\mathcal{H}}

\newcommand{\CW}{\mathcal{W}}

\newcommand{\BB}{\mathbf B}
\newcommand{\BR}{\mathbf R}
\newcommand{\BI}{\mathbf I}

\newcommand{\rd}{\mathrm d}
\newcommand{\LLL}{\RL^1(\log \RL)^1}

\newcommand {\bxi}{\boldsymbol \xi}

\newcommand {\RL}{\mathrm L}

\newcommand{\RW}{\mathrm W}
\newcommand{\RC}{\mathrm C}

\newcommand{\normalstek}[1]{\bar \sigma_{#1}}
\newcommand{\normallap}[1]{\bar \lambda_{#1}}
\newcommand{\optimstek}[1]{\Sigma_{#1}}
\newcommand{\optimlap}[1]{\Lambda_{#1}}
\newcommand{\wk}{\rightharpoonup}
\newcommand{\wks}{\xrightharpoonup{*}}

\newcommand{\abs}[1]{\left\lvert #1 \right\rvert}

\newcommand{\set}[1]{\left\{ #1 \right\}}

\newcommand{\norm}[1]{\left\| #1 \right\|}
\newcommand{\bo}\boldsymbol{}
\newcommand{\bigo}[2][]{O_{#1}\left( #2 \right)}
\newcommand{\smallo}[2][]{o_{#1}\left( #2 \right)}
\newcommand{\bone}{\bo{1}}

\newcommand{\ci}{(\textbf{M1})\xspace}
\newcommand{\cii}{(\textbf{M2})\xspace}
\newcommand{\ciii}{(\textbf{EF1})\xspace}
\newcommand{\civ}{(\textbf{EF1})\xspace}
\newcommand{\cv}{(\textbf{EF2})\xspace}
\newcommand{\cvi}{(\textbf{M3})\xspace}
\newcommand{\csi}{(\textbf{M1})\xspace}
\newcommand{\csii}{(\textbf{M2})\xspace}
\newcommand{\csiii}{(\textbf{EF1})\xspace}
\newcommand{\csiv}{(\textbf{EF1})\xspace}
\newcommand{\csv}{(\textbf{EF2})\xspace}
\newcommand{\csvi}{(\textbf{M3})\xspace}

\DeclareMathOperator{\dist}{dist}

\DeclareMathOperator{\diam}{diam}
\DeclareMathOperator{\defect}{def}

\renewcommand{\le}{\leqslant}
\renewcommand{\ge}{\geqslant}

\renewcommand{\tilde}{\widetilde}
\newcommand{\eps}{\varepsilon}
\renewcommand{\phi}{\varphi}

\renewcommand{\S}{\mathbb S}
\DeclareMathOperator{\capa}{cap}
\DeclareMathOperator{\Area}{Area}

\renewcommand{\bar}[1]{\overline{#1}}

\renewcommand{\div}{\operatorname{div}}
\title[Continuity of eigenvalues]{Continuity of eigenvalues and
shape optimisation for \\ Laplace and Steklov problems}

\author{Alexandre Girouard}
\address{D\'epartement de math\'ematiques et de statistique, Pavillon Alexandre-Vachon, Universit\'e Laval, Qu\'ebec, QC, G1V 0A6, Canada}
\email{alexandre.girouard@mat.ulaval.ca}

\author{Mikhail Karpukhin}
\address{Mathematics 253-37, Caltech, Pasadena, CA 91125, USA.}
\email{mikhailk@caltech.edu}

\author{Jean Lagac\'e}
\address{Department of Mathematics, University College London, Gower Street, London, WC1E 6BT, United Kingdom}
\email{j.lagace@ucl.ac.uk}

\begin{document}
\begin{abstract}
We associate a sequence of variational eigenvalues to any Radon measure on a
compact Riemannian manifold. For particular choices of measures, we recover the
Laplace, Steklov and other classical eigenvalue problems. In the first part of
the paper we study the properties variational eigenvalues and establish a
general continuity result, which shows for a sequence of measures converging in
the dual of an appropriate Sobolev space, that the associated eigenvalues converge as
well. The second part of the paper is devoted to various applications to shape
optimization. The main theme is studying sharp isoperimetric inequalities for
Steklov eigenvalues without any assumption on the number of
  connected components of the boundary. In particular,  we solve the
  isoperimetric problem for each Steklov  eigenvalue of planar domains: the best
  upper bound for the $k$-th perimeter-normalized Steklov eigenvalue is $8\pi
  k$, which is the best upper bound for the $k^{\text{th}}$ area-normalised  eigenvalue
  of the Laplacian on the sphere. The proof involves realizing a weighted
  Neumann problem as a limit of Steklov problems on perforated domains. For
  $k=1$, the number of connected boundary components of a maximizing sequence
must tend to infinity, and we provide a quantitative lower bound on the number
of connected components. A surprising consequence of our analysis is that any
maximizing sequence of planar domains with fixed perimeter must collapse to a
point.

\end{abstract}

\maketitle
\section{Introduction}

For a compact, connected Riemannian manifold $(M,g)$ of dimension $d$, with or
without Lipschitz boundary $\del M$, the Laplace eigenvalue problem consists in
determining all $\lambda \in \R$ for which the following eigenvalue problem
admits a nontrivial solution:
\begin{equation}
  \begin{cases}
    - \Delta_g u = \lambda u & \text{in } M, \\
    \del_n u = 0 & \text{on } \del M, \text{ when } \del M \ne \varnothing,
  \end{cases}
\end{equation}
where $\del_n u$ is the outward normal derivative of $u$. Similarly, when $\del
M \ne \varnothing$ the Steklov problem consists in determining all $\sigma \in
\R$ such that the following boundary value problem admits a nontrivial
solution:
\begin{equation}
  \begin{cases}
    \Delta_g u = 0 & \text{in } M, \\
    \del_n u = \sigma u & \text{on } \del M.
  \end{cases}
\end{equation}
The eigenvalues of these problems form nondecreasing sequences
\begin{equation}
  \begin{aligned}
    0 &= \lambda_0(M,g) < \lambda_1(M,g) \le \lambda_2(M,g) \le \dotso \nearrow
    \infty, \\
    0 &= \sigma_0(M,g) < \sigma_1(M,g) \le \sigma_2(M,g) \le \dotso \nearrow
    \infty,
  \end{aligned}
\end{equation}
where each eigenvalue is repeated according to multiplicity. For each $k \in
\N$, we investigate sharp upper bounds for $\lambda_k(M,g)$ and $\sigma_k(M,g)$.
To that end, we define the scale invariant quantities
\begin{equation}
  \normallap{k}(M,g) := \lambda_k(M,g) \Vol_g(M)^{\frac 2 d } \qquad \text{and}
  \qquad  \normalstek{k}(M,g) := \sigma_k(M,g) \frac{\CH^{d-1}(\del
  M)}{\Vol_g(M)^{\frac{d-2}{d}}}.
\end{equation}
Here, $\Vol_g(M)$ is the Riemannian volume of $M$ and $\CH^{d-1}_g$ is the
associated $(d-1)$-dimensional Hausdorff measure. These normalisations are
natural for both problems, see e.g. \cite{GNY2004,ceg2,fraschoen, GL}.

In the present paper we study the relation between
$\normalstek{k}(M,g)$ and
$\normallap{k}(M,g)$. In order to do so, we introduce the unifying framework of
\emph{variational eigenvalues associated with a Radon measure}. Given a Radon
measure $\mu$ on $M$, we define
\begin{equation}
  \label{eq:varifirst}
  \lambda_k(M,g,\mu) := \inf_{F_{k+1}} \, \, \, \sup_{f \in F_{k+1} \setminus
  \set 0} \frac{\int_M \abs{\nabla_g f}_g^2 \de v_g}{\int_M f^2 \de \mu},
\end{equation}
where $F_{k+1}$ is a $(k+1)$-dimensional subspace of $\RC^\infty(M) \cap
\RL^2(M,\mu)$.
To the best of our knowledge, variational eigenvalues for Radon measures were
first defined to describe Laplacians on fractal sets, see e.g. the survey of
Triebel \cite{triebelfractal}. In the context of spectral bounds and shape
optimisation,
they first appeared in the work of Kokarev~\cite{kok} as a relaxation of the
optimisation constaint. One should also see the
influencial work of Korevaar~\cite{kvr} and especially of
Grigor'yan--Netrusov--Yau~\cite{GNY2004} where the spectrum of energy forms is
investigated.
The variational eigenvalues admit a natural normalisation
\begin{equation}
  \normallap{k}(M,g,\mu) := \lambda_k(M,g,\mu)
  \frac{\mu(M)}{\Vol_g(M,g)^{\frac{d-2}{d}}}.
\end{equation}
One of the main interest of introducing these variational eigenvalues is that
they unify the presentation of several eigenvalue problems. For instance, for
$\mu=\rd v_g$ the volume measure associated to metric $g$, $\normallap{k}(M,g,\mu)=\normallap{k}(M,g)$, while for $\mu=\iota_* \rd
  A_g$, the pushforward by inclusion of the boundary measure, $\normallap{k}(M,g,\mu)=\normalstek{k}(M,g)$.

We present results of two types. On one hand, we study variational eigenvalues
on their own. In Section~\ref{section:variationaleigenvalues} we establish the
necessary functional analysis preliminaries. We define $p$-admissible measures,
which are essentially measures that  can be viewed as elements of the dual
space $(\RW^{1,p}(M))^*$ with certain compactness properties, and give various
examples of $p$-admissible measures. Section~\ref{variationalev:sec} is
concerned with the properties of variational eigenvalues. For example, we show
that the eigenvalues associated with a $2$-admissible measure form a discrete unbounded sequence
accumulating only at $+\infty$. The main result of this section is
Proposition~\ref{prop:weakenough}, which states that convergence of measures in
$(\RW^{1,p}(M))^*$ for appropriate values of $p$ implies convergence of the
eigenvalues.

On the other hand, we apply this continuity result
to obtain the aforementioned relations between $\normalstek k$ and $\normallap k$.
We start with isoperimetric inequalities when $d=2$, where the form of the results are
cleaner. We also obtain quantitative bounds for the first non-trivial eigenvalue
in terms of the number of boundary components and describe applications to the
free boundary minimal surfaces. Convergence results and isoperimetry for an arbitrary $d\ge 2$ are
formulated in Theorem~\ref{thm:homointro} below. We finish the introduction by
stating some results in spectral flexibility which follow from our results on
approximations.

\subsection{Optimal isoperimetric inequalities for surfaces}

Maximisation of Steklov eigenvalues normalised by perimeter goes back to the work of
Weinstock~\cite{wein}. He proved that for simply-connected planar
domains, $\normalstek{1}(\Omega)\leq 2\pi$, with equality if and only if $\Omega$ is a disk. 
This was followed by works of Hersch--Payne--Schiffer~\cite{hps}, then later
Girouard--Polterovich~\cite{gpsurface} and Karpukhin~\cite{Karpukhin2017}  who proved that
\begin{equation}\label{ineq:hps}
  \normalstek{k}(M)\leq 2\pi (k+\gamma+b-1),
\end{equation}
this time for compact surfaces $M$ of genus $\gamma$ with $b$ connected boundary components.
It follows from Girouard--Polterovich~\cite{GPHPS} that for $\gamma=0$ and $b=1$, this bound is saturated by a sequence of simply-connected domains $\Omega^\eps\subset\R^2$ that degenerates to a union of $k$
identical disks as $\eps\to 0$.
Bounds for $\normalstek{k}$ which do not depend on the number of boundary components
are notably more elusive.
For $M$ a compact orientable surface of genus $\gamma$ with boundary, it was proved by Kokarev~\cite{kok} that
\begin{equation}\label{eq:kokarevbound}
\normalstek{1}(M)\le 8\pi(\gamma+1).
\end{equation}

The bound~\eqref{eq:kokarevbound} was later generalized in~\cite{KS} in the
following way. The \emph{conformal eigenvalues} of a compact Riemannian manifold $(M,g)$ are defined as
$$
\optimlap{k}(M,[g]) = \sup_{h\in [g]}\normallap{k}(M,h).
$$
By the work of Korevaar~\cite{kvr}  $\optimlap{k}(M,[g])<+\infty$.
See also Hassannezhad~\cite{hassannezhad} and  Colbois--El Soufi~\cite{colboisElSoufi}.

\begin{theorem}[Karpukhin--Stern~\cite{KS}]
  \label{thm:KarpukhinStern}
  Given any closed Riemannian surface $(M,g)$ and any Lipschitz domain $\Omega\subset
M$, one has
\begin{equation}
\label{ineq:KS}
\normalstek{1}(\Omega,g)<\optimlap{1}(M,[g])\qquad
\text{and}\qquad
\normalstek{2}(\Omega,g)<\optimlap{2}(M,[g]).
\end{equation}
\end{theorem}
To obtain~\eqref{eq:kokarevbound} from~\eqref{ineq:KS} using the bound
$\normallap{1}(M,g) \le 8\pi (\gamma +1)$ from Yang--Yau.

The first result of the present paper is a non-strict version of~\eqref{ineq:KS} valid for all values of $k$.

  \begin{theorem}
\label{thm:sigmaklambdak}
Let $(M,g)$ be a compact Riemannian surface and let $\,\Omega\subset M$ be a smooth domain such that $\del\Omega\cap\del M$ is either empty or equal to $\del M$. Then  one has
\begin{equation}
\label{ineq:sigmaklambdak}
\normalstek{k}(\Omega,g)\le\optimlap{k}(M,[g]).
\end{equation}
This inequality is sharp for each $k$.
\end{theorem}
To prove Theorem~\ref{thm:sigmaklambdak} we construct a sequence of conformal
metrics $e^{f_n}g$ such that the metrics $\mu_n=e^{f_n}$ concentrate near the
boundary $\partial\Omega$. Those metrics are reminescent of
the construction of Lamberti--Provenzano \cite{LambPro} and
Arrieta--Jim\'enez-Casas--Rodr\'iguez-Bernal \cite{Arrieta2008}. We use our
continuity results to show directly that the
spectrum converges.
\begin{remark}
 The corresponding inequality with $\optimlap{k}(M,g)$ replaced with
  $\normallap{k}(M,g)$ is not true in general. This will be made explicit in
  Theorem \ref{thm:compsteklovneumann}. 

The proof of Theorem~\ref{thm:sigmaklambdak} is very different from that of Theorem~\ref{thm:KarpukhinStern}. It is simpler, works for all values of $k$ as well
as for surfaces with boundary, but
does not imply the strict inequality. 
\end{remark}

Sharpness of inequality~\eqref{ineq:sigmaklambdak} follows from the next theorem
for closed surfaces.
\begin{theorem}[Girouard--Lagac\'e~\cite{GL}]
 \label{GL:thm}
 For any closed Riemannian surface $(M,g)$ and any $k\ge 0$ there exists a sequence of domains $\Omega^\eps\subset M$ such that 
 $$
 \normalstek{k}(\Omega^\eps,g)\xrightarrow{\,\eps\to0\,}\normallap{k}(M,g).
 $$
\end{theorem}
Surfaces with boundary are treated in Theorem \ref{manifoldhomo:thm}. In either
cases, the domains $\Omega^\eps$ are obtained by removing small disks from $M$.
 To see that this theorem implies the sharpness of inequality
 \eqref{ineq:sigmaklambdak}, one can repeatedly apply it to
 a sequence of conformal metrics $g_n = e^{f_n} g$ such that $\normallap{k}(M,g_n) \to \optimlap{k}(M,[g])$.

The following notations allows us to clarify the statement of our results:
\begin{equation}
   \label{eq:defoptimal}
 \begin{aligned}
   \optimstek{k}(M,g)&=\sup_{\Omega\subset M}\normalstek{k}(\Omega,g), \\
 \optimlap{k}(M) &= \sup_g\normallap{k}(M,g),\\
 \optimstek{k}(M) &= \sup_{g,\,\Omega\subset M}\normalstek{k}(\Omega,g),
 \end{aligned}
\end{equation}
where $M$ is a compact surface and $\Omega$ is a Lipschitz domain.
For these optimal eigenvalues the results of this section can be summarized as follows.
\begin{cor}
  For any compact surface $M$, any conformal class $[g]$ on $M$ and any $k\ge 0$ one has
 $$
 \optimstek{k}(M,g) = \optimlap{k}(M,[g]),\quad 
 \optimstek{k}(M) = \optimlap{k}(M).
 $$
\end{cor}
 In particular, using the known results on the exact values of $\optimlap{k}(M)$
 obtained in~\cite{KNPP,karpRP2} respectively, 
 \begin{equation}
\label{eq:SigmakS2}
\optimstek{k}(\mathbb{S}^2) = 8\pi k,
\end{equation}
$$
\optimstek{k}(\mathbb{RP}^2) = 4\pi (2k+1).
$$

\subsection{Optimal isoperimetric inequalities for planar domains. }
Because any domain $\Omega\subset\R^2$ is diffeomorphic to a domain in the sphere $\mathbb{S}^2$, it follows from~\eqref{eq:SigmakS2} that
  $\normalstek{k}(\Omega,g_0)\leq 8\pi k$.
  Following the ideas of~\cite{GL} we show that this inequality remains sharp for planar domains. 
  \begin{theorem}\label{thm:mainSigmak}
  Let $\Omega\subset\R^2$ be a bounded simply-connected domain with Lipschitz
  boundary. There exists a sequence $\Omega^\eps\subset\Omega$ of subdomains,
  with $\partial\Omega\subset\partial\Omega^\eps$, such that
  $$
  \normalstek{k}(\Omega^\eps,g_0) = \sigma_k(\Omega^\eps,g_0) \CH^1(\partial\Omega^\eps)  \xrightarrow{\,\eps\to0\,}8\pi k.
  $$
  In particular, 
  $$
  \optimstek{k}(\R^2):=\sup_{\Omega \subset \R^2} \normalstek{k}(\Omega,g_0)=8\pi k.
  $$
\end{theorem}
The domains $\Omega^\eps$ are obtained by removing small disks from $\Omega$.
In particular, this solves~\cite[Open problem 2]{gpsurvey} for $d=2$.

\subsection{Quantitative isoperimetric bounds for
\texorpdfstring{$\normalstek{1}$}{the first Steklov eigenvalue}}

Following~\cite[Theorem 4.3]{fraschoen2}, it was suggested
in~\cite{gpsurvey} that the number of boundary components in a maximizing
sequence for $\optimstek{1}(\S^2)$  needs to be unbounded. Indeed, let $M_{0,b}$
be a a compact orientable surface of genus $0$ with $b$ boundary components and define
$$
\optimstek{1}(0,b) = \sup_g \, \, \normalstek{1}(M_{0,b},g).
$$
The monotonicity results of
\cite[Theorem 4.3]{fraschoen2} and \cite[Theorem 1.3]{matthiesenpetrides} imply that $\optimstek{1}(0,b)$ is strictly monotone in $b$. Thus, Theorem~\ref{GL:thm} yields 
$$
\optimstek{1}(0,1)<\ldots<\optimstek{1}(0,b)<\optimstek{1}(0,b+1)<\ldots\nearrow 8\pi,
$$
which confirms the claim, yielding the direct corollary.

\begin{cor}
  \label{bc:cor}
  Any sequence of surfaces $M^\eps$ of genus 0 such that $\normalstek{1}(M^\eps)\xrightarrow{\eps\to 0}\ 8 \pi$ has unbounded number of
  connected boundary components.
\end{cor}

We refine Corollary~\ref{bc:cor} and obtain at the same time a quantitative
improvement to Kokarev's bound~\eqref{eq:kokarevbound}.
\begin{theorem}\label{thm:quantbdry}
  For every $b \ge 1$ and every metric
  $g$ and $\eps > 0$, 
  \begin{equation}
    \label{eq:conditional}
  \normalstek{1}(M_{0,b},g)
  \le \max\set{4 \pi\frac{2 + \eps}{1 +\eps}, \frac{8\pi}{1 +
  \exp(-(1+\eps)b)}}.
\end{equation}
\end{theorem}
This theorem follows from the more general Theorem~\ref{thm:quantbdrycore}.
For planar domains, this more general result also leads to the following bound, which implies that any $\optimstek{1}(\R^2)$-maximising sequence of domains with fixed perimeter shrinks to a point.
\begin{theorem}\label{thm:quantdiameter}
  For every $\eps > 0$ and every $\Omega \subset \R^2$ be a connected bounded
  domain with smooth boundary,
  $$\normalstek{1}(\Omega)\le \max \set{4 \pi\frac{2 + \eps}{1+\eps},
    \frac{8\pi}{1 + \exp\left( \frac{-(1 +\eps)}{2}\frac{\CH^1(\del
  \Omega)}{\diam(\Omega)} \right)}}.$$
\end{theorem}

 \begin{remark}
It follows from the seminal work of Fraser--Schoen~\cite{fraschoen2} that
 free boundary minimal surfaces in the unit ball are intimately related to the
 maximal Steklov eigenvalues. In particular, given an embedded free boundary minimal
 surface in the unit ball, its coordinates are Steklov eigenfunctions with
 eigenvalue $\sigma=1$.  In~\cite[Conjecture 3.3]{fraserli}, Fraser and Li conjectured that for each free boundary minimal surfaces $M$ properly embedded in $\mathbb B \subset \R^3$, this Steklov eigenvalue is always the first one, so that in such
 a case $2 \Area_{g}(M) = \CH^1(\del M) = \normalstek{1}(M,g)$. We can read
 Theorem \ref{thm:quantbdrycore} in this setting.
Let $M_{0,b}$ be a free boundary minimal surface of genus $0$ with $b$
   boundary components properly embedded in $\B \subset \R^3$ by its first
   Steklov eigenfunctions.
 Then, for every $\eps > 0$,
   \begin{equation}
     \Area_g(M_{0,b}) \le \max \set{2\pi \frac{2+\eps}{1+\eps}, \frac{4\pi}{1 +
     \exp(-(1 + \eps)b)}}.
   \end{equation}
   Under the Fraser--Li conjecture, this holds for all free boundary minimal
   surfaces of genus 0 with $b$ connected boundary component that are properly embedded in the unit ball $\mathbb{B}\subset\R^3$.
In other words, if the Fraser--Li conjecture is true, free boundary minimal
surfaces with large area must have a large number of boundary components.
 \end{remark}

\subsection{Isoperimetry and homogenisation for domains in \texorpdfstring{$\R^d$}{Rd}.}

Theorem~\ref{thm:mainSigmak} is a consequence of a more general result valid for
domains in $\R^d$. Let $\Omega \subset \R^d$ be a domain,
that is a connected bounded open set
with Lipschitz boundary.
For $\beta : \Omega \to \R$ nonnegative and non trivial, consider the weighted
Neumann problem
\begin{equation}
  \begin{cases}
    -\Delta f = \lambda \beta f& \text{in } \Omega,\\
    \del_n f = 0 &\text{on } \del \Omega.
  \end{cases}
\end{equation}
We assume that $\beta \in \RL^{d/2}(\Omega)$ if $d \ge 3$, and
$\beta \in
\LLL(\Omega)$ if $d = 2$ (see p. \pageref{page:LLL} for the definition of this
space which contains $\RL^p$, $p > 1$). 
If the flat metric on $\R^d$ is denoted by $g_0$, then the eigenvalues of this
problem can be understood in the weak sense as in
\eqref{eq:varifirst},
as the
variational eigenvalues
$\lambda_k(\Omega,g_0,\beta \rd v_{g_0})$. 
\begin{theorem}
  \label{thm:homointro}
  For any domain $\Omega\subset\R^d$, and any nonnegative $0\not\equiv\beta \in
  \RL^{d/2}(\Omega)$ ($d \ge 3$) or $\beta \in \LLL(\Omega)$ ($d = 2$), there exists a family $\Omega^\eps
\subset \Omega$ of domains such that for each $k \in \N$,
\begin{equation}
  \label{eq:convergenceSteklov}
  \normalstek{k}(\Omega^\eps,g_0) 
  \xrightarrow{\eps \to 0} \normallap{k}(\Omega,g_0,\beta\rd v_{g_0}).
\end{equation}
For the same family $\Omega^\eps$,
\begin{equation}
  \lambda_k(\Omega^\eps,g_0) \xrightarrow{\eps \to 0} \lambda_k(\Omega,g_0),
  \qquad \text{and} \qquad \Vol_{g_0}(\Omega^\eps) \xrightarrow{\eps \to 0}
  \Vol_{g_0}(\Omega).
\end{equation}
\end{theorem}

We note that combining the methods of \cite{GHL,GL}, we could have proved a weaker
form of this result, i.e. with $\beta$ continuous, using an intermediate
dynamical boundary value problem. The proof that we present here is more direct,
allows for a more singular $\beta$, and gives more information on domains that are nearly maximizing $\normallap{k}$.

In order to prove this result, we realise the domains $\Omega^\eps$ by removing
tiny balls from $\Omega$ whose centres are periodically distributed. The
construction is in the spirit of homogenisation theory, with the distinction
that the radius of the balls removed is not uniform, but rather varies according
to the a continuous approximation of the function $\beta$, and is chosen so that the total boundary area tends to
$\infty$ in a controlled way as $\eps \to 0$. Variation within periods in homogenisation theory has
also been explored in \cite{ptashnyk}. In our method of proof, the number of
boundary components tends to $\infty$. By Theorem \ref{thm:quantbdry}, this is
unavoidable in dimension $2$ since we can obtain planar domains with
$\normalstek{1}$ as close to $8\pi$ as we want. In higher dimension, it is
possible to achieve the same result with only one boundary component, see
\cite{fraschoen3,GL}.

Finally, we remark that a straightforward modification of our method yields an
analogous result for compact Riemannian manifolds, see
Remark~\ref{manifold:remark} and a similar result for $\beta$ continuous on
closed Riemannian manifolds in \cite[Theorem 1.1]{GL}.
\begin{theorem}
\label{manifoldhomo:thm}
  For any compact Riemannian manifold $(M,g)$ of dimension $d$, and any
  nonnegative $0\not\equiv\beta \in \RL^{d/2}(M)$ ($d \ge 3$) or $\beta \in
  \LLL(M)$ ($d=2$), there exists a family $\Omega^\eps
\subset M$ of domains such that for each $k \in \N$,
\begin{equation}
  \label{eq:convergenceSteklovmanifolds}
  \normalstek{k}(\Omega^\eps,g) 
  \xrightarrow{\eps \to 0} \normallap{k}(M,g,\beta \rd v_g).
\end{equation}
Harmonic extensions of the associated eigenfunctions from $\Omega$ to $M$
converge strongly to eigenfunctions of the limit problem in $\RW^{1,2}(M)$. 

For the same family $\Omega^\eps$,
\begin{equation}
  \lambda_k(\Omega^\eps,g) \xrightarrow{\eps \to 0} \lambda_k(M,g),
  \qquad \text{and} \qquad \Vol_{g}(\Omega^\eps) \xrightarrow{\eps \to 0}
  \Vol_{g}(M).
\end{equation}
Here, harmonic extensions of the associated eigenfunctions converge weakly to
eigenfunctions of the limit problem in $\RW^{1,2}(M)$.
\end{theorem}

\subsection{Flexibility in the prescription of the Steklov spectrum.}
Bucur--Nahon~\cite{bucurnahon} have recently shown
  that the Weinstock and Hersch--Payne--Schiffer inequalities are unstable, in the sense that there are simply-connected domains that are very
  far from the disk --- or from a union of $k$ identical disks --- with their
  $k$th normalised
  eigenvalue arbitrarily close to $2\pi k$. 
  In fact, they prove the following result.
  \begin{theorem}[Bucur--Nahon, {\cite[Theorem 1.1]{bucurnahon}}]
    \label{thm:BucurNahon}
    Let $\Omega_1, \Omega_2 \subset \R^2$ be two bounded, conformally equivalent
     domains with smooth boundary. Then, there exists a sequence of open
    domains $\Omega^\eps$ with uniformly bounded perimeter such that
    \begin{equation}
      d_{\text{Haus}}(\del \Omega^\eps, \del \Omega_1) \xrightarrow{\eps \to 0} 0,
      \quad \text{and,} \quad \text{for all } k \in \N, \; 
      \normalstek{k}(\Omega^\eps) \xrightarrow{\eps \to 0} \normalstek{k}(\Omega_2).
    \end{equation}
  \end{theorem}
The domains $\Omega^\eps$ constructed in~\cite{bucurnahon} are diffeomorphic to the original domains. They
are obtained by a local perturbation of the boundary. We remark that a similar 
result can be obtained as an application of Theorem~\ref{thm:homointro}, see
Remark~\ref{rem:BucurNahon} for details. However, the domains $\Omega^\eps$
obtained this way have many small holes concentrated near the boundary
$\del\Omega_1$. 

We further investigate  flexibility results for the Steklov spectrum of domains
in Euclidean space. In many ways, the Neumann and Steklov problems have similar features. This
has led to an investigation of bounds for one eigenvalue problem in terms of the
other, see e.g. \cite{HassannezhadSiffert,kuttsigi1}.
It was previously thought that some universal inequalities between
perimeter-normalised Steklov eigenvalues and area-normalised Neumann eigenvalues
could exist.
It is known from
\cite[Section 2.2]{GPHPS} that normalised Steklov eigenvalues can
be arbitrarily small while keeping the normalised Neumann eigenvalues bounded away from
zero. We use Theorem \ref{thm:homointro} to prove that there are also domains
with arbitrarily small area-normalised Neumann eigenvalues
$\normallap{k}(\Omega,g_0)$, for which the Steklov eigenvalues are bounded away from
zero.
\begin{theorem}
  \label{thm:compsteklovneumann}
  There exists a sequence of planar domains $\Omega^\eps$ such that the
  normalised Steklov eigenvalue $\normalstek{1}(\Omega^\eps)\xrightarrow{\,\eps\to
  0\,}8\pi$ while for each $k\in\N$, the normalised Neumann eigenvalues satisfy
  $\normallap{k}(\Omega^\eps) \xrightarrow{\,\eps\to 0\,}0.$
\end{theorem}
The reader should compare with the results of~\cite[Section 5]{BucurHenrotMichetti}, where another family of examples where $\normalstek{1}(\Omega)\leq\normallap{1}(\Omega)$ fails is constructed.

  \begin{remark}
    Similarly to Theorem~\ref{thm:compsteklovneumann}, on any closed Riemannian surface $(M,g)$ there exists a sequence of conformal metrics $g_n=e^{f_n}g$ and a sequence of domains $\Omega_n\subset M$ such that $\normalstek{1}(\Omega_n,g_n)\xrightarrow{\eps\to 0}\Sigma_1(M,g)$ while for each $k\in\N$, the normalised Laplace eigenvalues satisfy
  $\normallap{k}(M,g_n) \xrightarrow{\,\eps\to 0\,}0.$
  \end{remark}

\subsection{Plan of the paper, heuristics, and strategies}
The majority of the paper is centred around
Theorems~\ref{thm:sigmaklambdak} and \ref{thm:homointro};
we either discuss their applications, develop the theory towards their proof and
justify some constraints that become apparent in the proof. We note that the
proof of both of these theorems are very similar in nature under the scheme that
we develop.

In Section~\ref{sec:applications}, we start by presenting applications of
Theorem~\ref{thm:homointro}, including the proofs of Theorems~\ref{thm:mainSigmak} and~\ref{thm:compsteklovneumann}. 
The proofs of Theorems~\ref{thm:quantbdry},~\ref{thm:quantdiameter} are
independent of the rest of the paper and are also presented there. 

In Section~\ref{section:variationaleigenvalues}, we present the general
framework of variational eigenvalues associated to a Radon measure. This
is a unifying framework which allows one to compare different, seemingly
unrelated, eigenvalue problems on a manifold. We start with a general
description of the setup and give conditions on and examples of measures giving rise to
eigenvalues behaving like Laplace eigenvalues. Finally, we obtain continuity
of the eigenvalues and eigenfunctions with respect to convergence of the
measures in the dual of some appropriate Sobolev space.

An immediate application of the framework presented in
Section~\ref{section:variationaleigenvalues}, is given
in Section~\ref{sec:1applications}. 
In particular, we prove that on any surface we can approximate Steklov-type
eigenvalues with Laplace eigenvalues associated with a degenerating sequence of
metrics, giving as an application a proof of Theorem~\ref{thm:sigmaklambdak}.

\subsection{Notation}

We make here a list of notation that is explicitly reserved throughout the
paper.

\subsubsection*{Manifolds and their domains}

Whenever we mention a manifold or a surface without qualification, it may have nonempty
boundary, which is always assumed to be Lipschitz. In any PDE
written in strong form, the boundary term may be ignored when the manifold has
empty boundary. Closed manifolds and surfaces are compact and without boundary. We
reserve the letter $M$ for manifolds. 
When $M$ has
nonempty boundary, we denote by $\operatorname{int}(M)$ the set $M \setminus
\del M$. 

A domain in a manifold $M$ is a bounded open connected subset of $M$ with
Lipschitz boundary if its boundary is nonempty. We reserve the letters $\Omega$
and $\Upsilon$ for domains.

\subsubsection*{Standard measures and metrics}
Let $(M,g)$ be a Riemannian manifold. 
We denote he volume measure $\rd v_g$. If
there is a canonical metric on $M$, it is denoted by $g_0$. This could be the
flat metric on $\R^d$ or the round metric on the sphere. It is usually a
constant curvature metric. If
$M$ has a boundary, we write $\rd A_g$ for the boundary measure induced by the
metric. It is often useful to recall that $\rd A_g := \CH^{d-1}{}_{\lfloor} \del
M$, where $\CH^{d-1}$ is the $(d-1)$-dimensional Hausdorff measure induced by
the metric $g$ on $M$. We abuse notation and make no distinction between $\rd
A_g$ as a measure on $\del M$, and the pushforward by inclusion $\iota_* \rd
A_g$ which is a measure on $M$.

In cases where confusion may arise, if we want to explicitly distinguish the restriction of $\rd v_g$ to a domain
$\Omega \subset M$ we write
\begin{equation}
  \rd v_g^\Omega := (\rd v_g)_{\rfloor} \Omega,
\end{equation}
and similarly for $\Sigma \subset \del M$,
\begin{equation}
  \rd A_g^\Sigma := (\rd A_g)_{\lfloor} \Sigma := \CH^{d-1}{}_\lfloor \Sigma.
\end{equation}

\subsubsection*{Standard function spaces and capacity}

Every vector space under consideration is defined over $\R$. For $X$ a
topological vector space, $\xi \in X^*, x \in X$ 
 we denote by $\langle \xi,x\rangle_X :=
\xi(x)$ the duality pairing. Since all vector spaces are real, we use this
notation to denote an inner product as well, without confusion.

For $p \in [1,\infty]$ we let $p'$ be its H\"older conjugate and for $p \in
[1,d)$ we  let $p^\star$ be its Sobolev conjugate, given respectively by 
\begin{equation}
  p' = \frac{p}{p-1} \qquad \text{and} \qquad p^\star = \frac{pd}{d-p}.
\end{equation}

In order to characterise critical scenarios in dimension $2$, we will require a
generalisation of the usual Lebesgue $\RL^p$ and Sobolev $\RW^{1,p}$ spaces. The
first spaces we introduce are the Orlicz spaces $\RL^p (\log \RL)^a$, for $p \ge
1$ and $a \in \R$ and $\exp \RL^a$ for $a > 0$. For a reference on Orlicz space,
see e.g. \cite[Chapters 4.6--4.8]{benettsharpley}. The space $\RL^p (\log
\RL^a)(M)$ consists of all \label{page:LLL}
  functions $f$ such that
  \begin{equation}
    \int_M \left[\abs{f} \left(\log (2+ \abs f)\right)^a\right]^p \rd v_g < \infty.
  \end{equation}
  For $p > 1$ and $a \in \R$, or $p = 1, a \ge 0$, it can be endowed with the
  Luxemburg norm
  \begin{equation}
    \norm{f}_{\RL^p (\log L)^a(M)} = \inf \set{ \eta > 0 : 
      \int_\Omega \Big[\abs{f/\eta} \left(\log (2+ \abs{f/\eta})\right)^a\Big]^p \rd v_g
    \le1},
  \end{equation}
  under which it is a Banach space. For $a > 0$, we also define the Orlicz spaces $\exp \RL^a$ to be
\begin{equation}
  \label{eq:exprla}
  \exp \RL^a := \set{f \in \RL^\infty(M) : \exists \eta > 0 : \int_M \exp\left(
  \abs{f/\eta}^a\right) \de v_g < \infty}.
\end{equation}
Just like the spaces $\RL^p (\log \RL)^a$, they can be endowed with the
Luxemburg norm
\begin{equation}
  \norm{f}_{\exp \RL^a} = \inf \set{\eta > 0 : \int_M \exp\left(\abs{f/\eta}^a
    \right) \de
  v_g \le 1},
\end{equation}
under which it is also a Banach space. The space $\exp L^1$ serves as a pairing space for
$\LLL$, see \cite[Theorem 4.6.5]{benettsharpley}, in the sense that there is $C
> 0$ so that for $f \in
\LLL$, $\phi \in \exp L^1$,
\begin{equation}
  \int f \phi \de v_g \le C \norm{f}_{\LLL} \norm{\phi}_{\exp \RL^1}.
\end{equation}
We identify $\exp \RL^1$ witht he dual of $\LLL$. For every $p \ge 1$ and $a,
\eps > 0$, 
we have the relations 
$$
\RL^\infty(M) \subset \exp \RL^a(M) \subset \RL^p(M)
$$
and
\begin{equation}
  \RL^{p+\eps}(M) \subset \RL^p (\log \RL)^{a}(M) \subset \RL^p(M) \subset
  \RL^p (\log \RL)^{-a} (M) \subset \RL^{p-\eps}(M).
\end{equation}
We also define for $p \ge 1$ and $a \in \R$ the
Orlicz--Sobolev spaces $\RW^{1,p,a}(M)$
as
\begin{equation}
  \label{eq:orliczsob}
  \RW^{1,p,a}(M) := \set{f \in \RL^p( \log \RL)^a(M) : \nabla f \in \RL^p(\log
  \RL)^a(M)}
\end{equation}
with the gradient being understood in the weak sense, see \cite[Section
2]{cianchi} for this definition. We note that for every $p \ge 1, a \ge 0, \eps >
0$ we have the relations
\begin{equation}
  \RW^{1,p+\eps}(M) \subset \RW^{1,p,a}(M) \subset \RW^{1,p}(M) \subset
  \RW^{1,p,-a}(M) \subset \RW^{1,p-\eps}(M).
\end{equation}

Finally, we will make use of the notion of $p$-capacity.
Given two sets
$\Upsilon \subset \subset \Omega \subsetneq M$,
we write
$$\RC_0^\infty(\Omega):=\set{f\in\RC^\infty(\Omega)\,:\, f \equiv 0 \text{ on }
\partial\Omega\cap \operatorname{int}(M)}.$$
The $p$-capacity of $\Upsilon$
with respect to $\Omega$ is defined as 
\begin{equation}
  \capa_p(\Upsilon,\Omega) := \inf \set{\int_M \abs{\nabla f}^p \de v_g : f \in
  \RC_0^\infty(\Omega), f \equiv 1 \text{ on } \Upsilon},
\end{equation}
and the $p$-capacity of $\Upsilon$ as
\begin{equation}
  \capa_p(\Upsilon) := \inf \set{\capa_p(\Upsilon,\Omega) : \Omega \subsetneq M,
  0 < \Vol_g(\Omega) \le \Vol_g(M)/2}.
\end{equation}
We note that if $\Omega \cap \del M$ is not empty, we do not require in the
definition of the capacity that $f$ vanishes on that set.

\subsubsection*{Asymptotic notation}

We make extensive use throughout the paper of the so-called Landau asymptotic
notation. We write
\begin{itemize}
  \item[$\bullet$] without distinction, $f_1 = \bigo{f_2}$ or $f_1 \ll f_2$ to mean that
    there exists $C > 0$ such that $\abs{f_1}
    \le C f_2$;
  \item[$\bullet$] $f_1 \asymp f_2$ to mean that $f_1 \ll f_2$ and $f_2 \ll f_1$;
  \item[$\bullet$] $f_1 \sim f_2$ to mean that $\frac{f_1}{f_2} \to 1$;
  \item[$\bullet$] $f_1 = \smallo {f_2}$ to mean that $\frac{f_1}{f_2} \to 0$.
\end{itemize}
The limits in the last two bullet points will either be as some parameter tends
to $0$ or $\infty$ and will be clear from context. The use of a subscript in the
notation, e.g. $f_1 \ll_M f_2$ or $f_1 = \smallo[k]{f_2}$, indicates that the constant
$C$, or the quantities involved in the definition of the limit may depend on the
subscript.

\subsection*{Acknowledgements}
The authors would like to thank Iosif Polterovich for introducing them to
spectral geometry.
This project stemmed from discussions held during the online miniconference on
sharp eigenvalue estimates for partial differential operators held by Mark
Ashbaugh and Richard Laugesen in lieu of a session at the AMS Sectional Meeting.
The authors would like to thank Pier Domenico Lamberti for reading an early
version of this manuscript and providing helpful comments, as well as Bruno Colbois for pointing out reference~\cite{Anne1986}.
AG is supported by NSERC and FRQNT.
The research of JL was supported by EPSRC grant EP/P024793/1 and the NSERC
Postdoctoral Fellowship. 

\section{Applications and motivation}
\label{sec:applications}

In this section, we give application of Theorem \ref{thm:homointro} to shape
optimisation for the Steklov problem in $\R^2$, and to spectral flexibility. We
also provide the proofs of Theorems \ref{thm:quantbdry} and \ref{thm:quantdiameter}.

\subsection{Approximation by Steklov eigenvalues}
\label{sec:largeplane}
We start by proving Theorem \ref{thm:mainSigmak} from Theorems
\ref{thm:sigmaklambdak} and
\ref{thm:homointro}. 
\begin{proof}[Proof of Theorem \ref{thm:mainSigmak}]
  Let $\Omega \subset \R^2$ be a simply-connected Lipschitz domain. 
  We know from \cite{hersch, PetridesSphere, KNPP} that $\optimlap{k}(\S^2) = 8\pi k$. Let $\delta > 0$, and $g$ be a smooth metric on $\S^2$ such that such that
\begin{equation}
  \normallap{k}(\S^2,g) > \optimlap{k}(\S^2) - \delta = 8 \pi k - \delta.
\end{equation}
Let $\Upsilon$ be $\S^2$ with a small disk removed.
It is well known
that as the radius of that disk goes to $0$, the Neumann eigenvalues $\lambda_k(\Upsilon,g)$
converge to $\lambda_k(\S^2,g)$, see~\cite[Théorème 2]{Anne1986}.
Thus, removing a small enough disk ,
$$\normallap{k}(\Upsilon,g)
 > \normallap{k}(\S^2,g) - \delta.$$
Let $\Phi : \Omega \to \Upsilon$ be a conformal diffeomorphism. Since Dirichlet energy
is a conformal invariant, the $k$th Neumann eigenvalue of $\Upsilon$ is equal to
the variational eigenvalue $\lambda_k(\Omega,g_0,\Phi^*(\rd v_g))$.
The homogenisation Theorem \ref{thm:homointro} guarantees the existence of
$\Omega^\eps \subset \Omega$ such that 
\begin{equation}
  \label{eq:choice}
  \sigma_k(\Omega^\eps) \CH^1(\del \Omega^\eps) > 
  \lambda_k(\Omega,g_0,\abs{\rd\Phi}^2 \rd x)  \int_{\Omega} \Phi^*(\rd v_g) -
  \delta.
\end{equation}
Putting this all back together yields the bound $\sigma_k(\Omega^\eps) \CH^1(\del
\Omega^\eps) > 8\pi k - 3\delta. $
Since $\delta > 0$ is arbitrary  $\optimstek{k}(\R^2) \ge 8\pi k$, and by
Theorem \ref{thm:sigmaklambdak} this is in fact an equality.
\end{proof}

The exact same proof can be used to obtain the comparison between Steklov and
Neumann eigenvalues.

\begin{proof}[Proof of Theorem \ref{thm:compsteklovneumann}]
  For $\delta > 0$, proceed as in the proof of Theorem \ref{thm:mainSigmak}, but
  start with $\Omega \subset \R^2$ such that $\normallap{k}(\Omega) < \frac
  \delta 2$, for instance a very thin rectangle. By Theorem \ref{thm:homointro}, one can choose $\eps$ in \eqref{eq:choice} small enough
  that $\normallap{k}(\Omega^\eps,g_0) < \delta$. This concludes the proof.
\end{proof}

\subsection{Geometric and topological properties of maximising
sequences}
\label{sec:diameterbound}
In the present section we prove Theorem~\ref{thm:quantbdry} and Theorem~\ref{thm:quantdiameter}

The domains
$\Omega^\eps$ constructed in Theorem~\ref{thm:homointro}, are
obtained by removing many tiny balls whose total boundary length tends to $+\infty$.
In particular, the length of each boundary
component relative to the total length of the boundary tends to zero.
We show that {\em any} maximizing sequence of domains for $\optimstek{1}(\S^2)$ or
$\optimstek{1}(\R^2)$ exhibits this behaviour. Moreover, for any metric on $M_{0,b}$
 one has the following quantitative relation between the relative length of the
 longest boundary component and the \emph{Steklov spectral defect}
 $$\defect(M_{0,b},g) := 8\pi
 - \normalstek{1}(M_{0,b},g).$$
\begin{theorem}\label{thm:quantbdrycore}
  \label{thm:defect}
  Let $(M,g)$ be a compact Riemannian surface of genus $0$,
  and let $L$ be the length of its longest boundary component. Then,
  \begin{equation}
    \label{eq:defectbound}
    \frac{\CH^1(\del M)}{L} \ge \left( 1 - \frac{\defect(M,g)}{4\pi}
    \right)_+ \log\left( \frac{8\pi}{\defect(M,g)} - 1 \right).
  \end{equation}
\end{theorem}
Here, for any real valued function $f$, we write $f_+ := \max(f,0)$.
One can interpret this result as a quantitative 
improvement of Kokarev's 
estimate~\eqref{eq:kokarevbound}. This is the essence of Theorem~\ref{thm:quantbdry} and Theorem~\ref{thm:quantdiameter} which we now prove using Theorem~\ref{thm:quantbdrycore}
\begin{proof}[Proof of Theorem \ref{thm:quantbdry}]
  The desired inequality holds unconditionally for every $\eps > 0$ when $\bar \sigma_1 := \bar
  \sigma_1(M_{0,b},g) \le 4 \pi$. We may therefore  assume without loss of generality that $\bar
  \sigma_1 > 4\pi$ and $b \ge 3$. 
  If $L$ is the length of the
  longest boundary component of a surface $M_{0,b}$ with $b$
  boundary components, then $\CH^1(\del M_{0,b}) \le Lb$. Exponentiating both
  sides in \eqref{eq:defectbound} and rearranging yields
  \begin{equation}
    \label{eq:rearr}
    \bar \sigma_1 \le \frac{8 \pi}{1 + \exp\left(\frac{-4\pi b}{(\bar \sigma_1 -
    4 \pi)}\right)}.
  \end{equation}
  We see that when $8 \pi > \bar \sigma_1 \ge 4\pi \frac{2 + \eps}{1 + \eps}$,
  then we get the upper bound
  \begin{equation}
    \bar \sigma_1 \le \frac{8 \pi}{1 + \exp\left(\frac{-4\pi b}{(\bar \sigma_1 -
    4 \pi)}\right)}
    \le \frac{8\pi}{1 + \exp(-(1 + \eps)b)},
  \end{equation}
  which gives \eqref{eq:conditional}. 
\end{proof}
\begin{proof}[Proof of Theorem \ref{thm:quantdiameter}]
  Let $\Omega$ be a connected bounded domain in $\R^2$, and $\CC \subset \del \Omega$
  be the boundary of the unbounded connected component of $\R^2 \setminus
  \Omega$. Then, we have that $2 \diam(\Omega) \le \CH^1(\CC) \le L$, where
  $L$ is the length of the longest boundary component. 
  The proof is completed in exactly the same way as above.
\end{proof}

The proof of Theorem \ref{thm:defect} is based on Hersch's renormalisation
scheme \cite{hersch}, as well as on a quantitative version of Kokarev's no atom
lemma \cite[Lemma 2.1]{kok}.

Let $\B$ be the unit ball in $\R^3$. For $\xi \in \B$, Hersch's conformal
diffeomorphism $\Psi_\xi : \S^2 \to \S^2$ is defined as
\begin{equation}
  \Psi_\xi(x) := \frac{(1 - \abs \xi^2) x + 2(1 + \xi \cdot x) \xi}{\abs{\xi +
  x}^2}.
\end{equation}

\begin{lemma}[Hersch's renormalisation scheme, see \cite{GNP,laugesen}]
  \label{lem:herschrenorm}
  Let $\mu$ be a measure on $\S^2$ such that for all $x \in \S^2$, $\mu(\set x)
  \le \frac 1 2 \mu(\S^2)$. Then, there exists
  a unique
  $\xi \in \B$ such that the pushforward measure $(\Psi_\xi)_* \mu$ has its
  center of mass at the origin. In other words, for $j \in \set{1,2,3}$, the coordinate functions $x_j :
  \S^2 \to \R$ satisfy
  \begin{equation}
    \label{eq:herschortho}
    \int_{\S^2} x_j \de (\Psi_{\xi})_* \mu = \int_{\S^2} x_j \circ \Psi_\xi \de
    \mu = 0.
  \end{equation}
\end{lemma}

\begin{remark}
In the classical formulation of Hersch's scheme as in e.g.~\cite{GNP} the measure $\mu$ is precluded from having points of non-zero mass. In the form presented here the measure $\mu$ is allowed to have point masses. The proof is different from the classical topological arguments and can be found in~\cite{laugesen}. 
\end{remark}

Given $y \in \S^2$, we define the closed hemisphere
\begin{equation}
  \S^2_y := \set{x \in \S^2 : x \cdot y \ge 0}.
\end{equation}
For $\Omega \subset \S^2_y$, recall that we define the capacity of $\Omega$ in
$\S^2_y$ as
\begin{equation}
  \capa_2(\Omega,\S^2_y) = \inf \set{ \int_{\S^2_y} \abs{\nabla f}^2 \de v_g : f
    \in \RC^\infty_0(\S^2_y), \, \,
  f\big|_{\Omega} \equiv 1}.
\end{equation}

\begin{lemma}
  \label{lem:capcap}
  Let $K_a\subset\S^2_y$ be a closed spherical cap of area $a < 2\pi$ centred at
  $y\in\S^2$. The
  capacity of $K_a$ in $\S^2_y$ is given by
  $$\capa_2(K_a,\S^2_y)=
  \frac{4\pi}{\log(\frac{4\pi}{a}-1)}.$$
\end{lemma}
\begin{proof}
  Let $\Phi:\D\to\S^2$ be the stereographic parametrisation of $\S^2_y$. By
  elementary trigonometry, $\Phi^{-1}(K_a) = B(0,r_a) \subset \D$, where
$$r_a=\sqrt{\frac{a}{4\pi-a}}.$$
Let $\chi_a:\D\rightarrow\R$ be the capacitary potential for $B(0,r_a)$, i.e.
the radial function defined in polar coordinates $(t,\theta)$ as
  $$
  \chi_a(t):= \begin{cases}
    \frac{\log t}{\log r_a} & \text{for } r_a < t \le 1 \\
    1 & \text{for } 0 \le t \le r_a.
  \end{cases}
  $$
  It follows by invariance of the Dirichlet energy under conformal
  transformations that $(\Phi^{-1})^* \chi_a$ is the capacitary potential of
  $K_a$, and thus that
  $$\capa_2(K_a,S^2_y) = \int_\D \abs{\nabla \chi_a}^2\,dv_{g_0}
  =\int_{r_a}^1 \left(\del_t \chi_a(t)^2\right) t \de t
  =\frac{4\pi}{\log(\frac{4\pi}{a}-1)}.$$
\end{proof}
\begin{proof}[Proof of Theorem \ref{thm:defect}]
 The proof is based on Kokarev's proof of \eqref{eq:kokarevbound},
  keeping a precise track of all quantities involved.
  Note that the theorem is trivially true when $\defect(M,g) \ge 4\pi$, so
  we assume  that $\defect(M,g) < 4\pi$. 
Let $\CC\subset \del M$ be the longest connected component of the boundary and
fix $y \in \S^2$. It follows from the Koebe uniformization theorem that there exists a diffeomorphism $\Phi : 
  M\to \Omega\subset \S^2_y$, conformal in the interior of $M$, sending $\CC$ to the equator,
  i.e $\Phi(\CC) = \del\S^2_y$.
  Let $\mu:=\Phi_* \rd s.$ be the pushforward of the  boundary measure by $\Phi$.
  The equator carries the length of $\CC$:
  $$\mu(\del \S^2_y)= \CH^1(\CC)\geq \frac{\CH^1(\partial M)}{b}.$$
We apply the Hersch renormalisation scheme to the measure $\mu$. By Lemma
\ref{lem:herschrenorm}, there is a unique $\xi \in \B$ so that the measure $\zeta := (\Psi_\xi)_*
\mu$ has its center of mass at the origin. In other words, we can read from
\eqref{eq:herschortho} that for $j \in \set{1,2,3}$, the functions $x_j \circ
\Psi_\xi \circ \Phi$ are trial functions for $\sigma_1$ on $M$.
Thus, by conformal invariance of the Dirichlet energy,
\begin{equation}
  \sum_{j = 1}^3 \sigma_1(M,g) \int_{\del M} x_j^2 \circ \Psi_\xi
  \circ \Phi \de s \le \sum_{j=1}^3 \int_{\Psi_\xi(\S^2_y)} \abs{\nabla_{g_0}
  x_j}^2 \de A_{g_0}.
\end{equation}
Using the pointwise identities $\sum_{j=1}^3x_j^2=1$ and  $\sum_{j=1}^3 \abs{\nabla_{g_0} x_j}^2 =2$,
  this leads to a strict form of Kokarev's bound from \cite{kok}:
\begin{equation}
  \normalstek{1}(\Omega,g) \le 2 \Area_{g_0}(\Psi_\xi(\S^2_y)) < 8\pi.
\end{equation}
  Because the total area of $\S^2$ is $4\pi$, it follows that the opposite
  hemisphere $\S^2_{-y}$ is mapped by $\Psi_{\xi}$ to a spherical cap with small area:
  \begin{equation}
  \begin{aligned}
    \Area_{g_0}\left(\Psi_{\xi}(\S^2_{-y})\right)
    &
  \leq
\frac{1}{2}\left(8\pi-\normalstek{1}(\Omega,g)\right)=\frac{\defect(\Omega,g)}{2}.
  \end{aligned}
\end{equation}
Let $z\in\S^2$ be the center of the spherical cap $K_a =\Psi_{\bxi}(\S^2_{-y})$,
where $a = \Area_{g_0}(K_a)$. The
center of the circle $\partial K_a$ is $\kappa z\in\B$, where
$2\pi(1-\kappa)=a<\defect(\Omega,g)/2.$
The spectral defect is smaller than $4\pi$ by hypothesis. Hence,
$$\kappa > 1-\frac{\defect(\Omega,g)}{4\pi}>0.$$
Let $\pi_{z}:\R^3\to\R$ correspond to the projection on the subspace $\R z$. That is,
$\pi_{z}(x):=(x\cdot z)z.$
Then the measure
$\rho:= (\pi_z)_* \zeta = (\pi_{z} \circ \Psi_\xi)_* \mu$ is supported in the interval $(-1,1)$ and
has an atom of weight $\mu(\del S^2_y)= \CH^1(\CC)$ located at $\kappa\in
(0,1)$.
Because the center of mass of $\zeta$ is the origin $0\in\B$, we have
$$0=\int_{-1}^1t \de \rho\geq\int_{-1}^0 t \de \rho+\kappa\rho(\{\kappa\})
=\int_{-1}^0 t \de \rho+\kappa\CH^1({\CC}).$$
In particular
$$\kappa\CH^1(\CC)\leq\int_{-1}^0-t \de\rho<\rho(-1,0)=\zeta(\S^2_{-z}).$$
It follows that
\begin{gather}
  \label{ineq:oppositecap}
  \zeta(\S^2_{-z})\geq\left(1-\frac{\defect(\Omega,g)}{4\pi}\right)\CH^1(\CC)>0.
\end{gather}
Let $\chi_a \in \RW^{1,2}(\S^2)$ be the capacitary potential of
$K_a\subset\S^2_{z}$, 
and $m_{\chi_a} =\frac{1}{\CH^1(\del M)}\int_{\S^2} \chi_a \de \zeta$.  We can
thus use $\chi_a - m_{\chi_a} \in \RW^{1,2}(\S^2)$ as a trial
function for $\sigma_1(M)=\lambda_1(\Omega,g_0,\mu)$. By Lemma \ref{lem:capcap}
$$\sigma_1(M)\int_{\S^2}(\chi_a - m_{\chi_a})^2 \de \zeta
\leq
\frac{4\pi}{\log(\frac{4\pi}{a}-1)}.$$
Now, using that $\xi_a\equiv 1$ on $K_a$ together with \eqref{ineq:oppositecap} we get
\begin{equation}
\begin{aligned}
  \int_{\S^2}(\chi_a - m_{\chi_a})^2\de \zeta
  &\geq
  \int_{K_a}(\chi_a-m_{\chi_a})^2\de \zeta +\int_{\S^2_{-z}}(\chi_a-
  m_{\chi_a})^2\de \zeta\\
  &\geq
  \left((1-m_{\chi_a})^2+
  m_{\chi_a}^2\left(1-\frac{\defect(M,g)}{4\pi}\right)\right)\CH^1(\CC)\\
  &=\left(\left(2-\frac{\defect(M,g)}{4\pi}\right)m_{\chi_a}^2-2m_{\chi_a}+1\right)\CH^1(\CC)\\
  &=\left(\frac{\normalstek{1}(M,g)}{4\pi} m_{\chi_a}^2-2m_{\chi_a}+1\right)\CH^1(\CC)\\
  &\geq
  \left(\frac{\normalstek{1}(M,g)-4\pi}{\normalstek{1}(M,g)}\right)\CH^1(\CC),    
\end{aligned}
\end{equation}
where in the last step we have minimized the quadratic form.
Putting all of this together leads to
  \begin{equation}
    \frac{\CH^1(\del M)}{\CH^1(\CC)} \ge \left( 1 - \frac{\defect(\Omega,g)}{4\pi}
    \right) \log\left( \frac{4\pi}{a} - 1 \right).
  \end{equation}
  Recall that $a=\Area_{g_0}(K_a) \leq \defect(\Omega,g)/2$ to finish the proof.
\end{proof}

\section{Admissible measures and associated function spaces}
\label{section:variationaleigenvalues}

The goal of this section is to properly define which measures allow for the
definition of variational eigenvalues, and to define associated Sobolev-type
spaces appropriate for our purpose. At the end of this section, we will provide
explicit examples of admissible measures.

\subsection{Sobolev-type spaces}

\begin{defi}
  For $1 \le p < \infty$, $M$ a compact Riemannian manifold and $\mu$ a Radon measure on $M$, we define $\CW^{1,p}(M,\mu)$ to be the completion of $\RC^\infty(M)$ with respect to the norm 
\begin{equation}
  \label{eq:completion}
  \norm u^p_{\CW^{1,p}(M,\mu)} = \int_M \abs u^p\de \mu + \int_M |\nabla u|_g^p\de v_g =
  \norm u^p_{\RL^p(M,\mu)} + \norm{\nabla u}_{\RL^p(M,g)}^p.
\end{equation}
This completion \eqref{eq:completion} gives rise to an embedding
$\tau^{\mu}_p:\mathcal{W}^{1,p}(M,\mu)\rightarrow \RL^p(M,\mu)$ of norm $1$.
\end{defi}
In the classical setting where $\mu$ is the
volume measure associated to $g$, the map $\tau^{\mu}_p$ is the natural embedding of the Sobolev space
$\RW^{1,p}(M)\subset \RL^p(M)$. If we
want to make $M$ explicit, we denote the embedding operator $\tau^{\mu}_{p,M}$.
Since $M$ is compact $\CW^{1,p}(M,\mu) \subset \CW^{1,q}(M,\mu)$ whenever $p \ge
q$. For $1 < p < \infty$, the closed unit ball in $\CW^{1,p}(M,\mu)$ is clearly
weakly compact so that $\CW^{1,p}(M,\mu)$ is a reflexive Banach space.

\begin{convention}
We adopt the following conventions in order to make the notation a bit lighter for spaces and operators that appear
often. We write $\RL^p(M)$ for $\RL^p(M,\rd v_g)$, $\RL^p(\del M)$ for
$\RL^p(M,\rd A_g)$ and $W^{1,p}(M) := \CW^{1,p}(M,\rd v_g)$. In
general, the measure $\mu$ may be omitted from the notation when it is the
natural volume measure given by the Riemannian metric, for instance as
$\lambda_k(M,g) := \lambda_k(M,g,\rd v_g)$. 
\end{convention}

  Denote the average of a function $f \in \RL^1(M,\mu)$ by
\begin{equation}
  \label{eq:faverage}
  m_{f,\mu} := \frac{1}{\mu(M)}\int_M f \de \mu.
  \end{equation}
\begin{defi}
  We say that a Radon measure $\mu$ supports a $p$-Poincar\'e inequality if there is
  $K > 0$ such that for all $f \in \CW^{1,p}(M,\mu)$
  \begin{equation}
    \int_M (f - m_{f,\mu})^p \de \mu \le K \int_M \abs{\nabla f}^p \de v_g.
  \end{equation}
  We denote by $K_{p,\mu}$ the smallest such number $K$.
\end{defi}

For general measures, the space $\CW^{1,p}(M,\mu)$ could be very different from the Sobolev space
$\RW^{1,p}(M)$ and solutions to (weak) elliptic PDEs in those spaces could lack
the natural properties one expects from them. For that reason we restrict ourselves
to a particular class of admissible measures, first introduced in \cite{KS} for
$d = p = 2$, see
also~\cite{kok} for a similar definition.

\begin{defi}
  \label{def:admissible}
Let $M$ be a Riemannian manifold, $p \in (1,\infty)$, and $\mu$ be a Radon measure on $M$ not
  supported on a single point. 
  The  measure $\mu$ is called $p$-admissible if it supports a $p$-Poincar\'e inequality and the operator
$\tau^{\mu}_{p}$ is compact. For $p = 2$, we simply say that $\mu$ is admissible.
\end{defi}

It is clear from that definition that $\rd v_g$ and the
boundary measure $\rd A_g$ are $p$-admissible for all $p \in (1,2]$. 
The
aim of the rest of this subsection is to prove the following two theorems. The
first one gives a characterisation of $p$-admissible measures. The second one
essentially  says that when $\mu$ is a $p$-admissible
measure there is an isomorphism between $\CW^{1,p}(M,\mu)$ and $\RW^{1,p}(M)$.
Their proofs are intertwined but they are better stated separately for ease of
reference.
\begin{theorem}
  \label{thm:charac}
  Let $\mu$ be a Radon measure and $p \in (1,2]$. Then, $\mu$ is $p$-admissible if and only if 
  the
  identity map on $\RC^\infty(M)$ extends to a compact operator $T_p^\mu :
  \RW^{1,p}(M) \to \RL^p(M,\mu)$.
\end{theorem}

\begin{theorem}
  \label{thm:normeq}
  Let $p \in (1,\infty)$ and suppose that $\mu$ is not supported on a single
  point and supports a $p$-Poincar\'e
  inequality.
  There exists $c_{p,\mu}$, $C_{p,\mu} > 0$ so that for every $f \in \RC^\infty(M)$
  \begin{equation}
    c_{p,\mu} \norm{f}_{\CW^{1,p}(M,\mu)} \le \norm{f}_{\RW^{1,p}(M)} \le
    C_{p,\mu} \norm{f}_{\CW^{1,p}(M,\mu)}.
  \end{equation}
  In particular, the completions $\CW^{1,p}(M,\mu)$ and $\RW^{1,p}(M)$ of \/
  $\RC^\infty(M)$ are isomorphic.
\end{theorem}

We start by proving the first inequality in Theorem \ref{thm:normeq}.
\begin{prop}
  \label{prop:babysteps}
  Let $p \in (1,\infty)$ and $\mu$ be a Radon measure on $M$ supporting a $p$-Poincar\'e inequality. Then, 
  there is $c_{p,\mu}>0$ such that
  \begin{equation}
    c_{p,\mu} \norm f_{\CW^{1,p}(M,\mu)} \le \norm{f}_{\RW^{1,p}(M)}.
  \end{equation}
In particular, the identity on $\RC^\infty(M)$ extends to a bounded operator
$T_p^\mu : \RW^{1,p}(M) \to \RL^p(M,\mu)$. 
\end{prop}
\begin{proof}
  We proceed in a similar manner to the proof of \cite[Lemma 2.2]{kok} where $d
  = p = 2$ and $\mu$ is a probability measure. For any $\Omega \subset M$ with
  $\mu(\Omega) > 0$, define $K_{p,*}(\Omega)$ via
  \begin{equation}
    \frac{1}{K_{p,*}(\Omega)} := 
    \inf_{\substack{\operatorname{supp}(f) \subset \Omega \\ f \not \equiv
    0}} \frac{\int_{\Omega} \abs{\nabla f}^p \de v_g}{\int_\Omega \abs f^p \de
    \mu}.
  \end{equation}
  Let $f$ be a smooth function supported on $\Omega$ and assume that
  $\mu(\Omega)^{p/p'}\mu(M)^{-p} \le 2^{-p}$. From this assumption and H\"older's
  inequality,
  \begin{equation}
    \begin{aligned}
      \int_{M} \abs{f - m_{f,\mu}}^p \de \mu & \ge 2^{1 - p} \int_M \abs{f}^p
      \de \mu - \int_M \abs{m_{f,\mu}}^p \de \mu \\ &\ge \left( 2^{1 - p} -
      \frac{\mu(\Omega)^{p/p'}}{\mu(M)^{p-1}} \right)
      \int_{\Omega} \abs f^p \de \mu
      \\
      &\ge 2^{-p} 
      \int_{\Omega} \abs f^p \de \mu.
    \end{aligned}
  \end{equation}
  We therefore have that for such $\Omega$
  \begin{equation}
    \frac{1}{K_{p,\mu}} \le
    \inf_{\substack{\operatorname{supp}(f) \subset \Omega \\ f \not \equiv
    0}} \frac{\int_{\Omega} \abs{\nabla f}^p \de v_g}{\int_\Omega \abs{f -
    m_{f,\mu}}^p \de \mu} \le \frac{2^{p}}{K_{p,*}(\Omega)}.
  \end{equation}
   Let $\set{\Omega_j}$ be a finite covering of $M$ with domains such that
  \begin{equation}
    0 < \mu(\Omega_j) < \frac{\mu(M)^{p'}}{2^{p'}}.
  \end{equation}
  with associated smooth partition of unity $\set{\rho_j^{p}}$. Then, for all $f \in \RC^\infty(M)$, 
  \begin{equation}
    \int_M \abs{f \rho_j}^p \de \mu \le 2^{2p-1} K_{p,\mu} \left( \int_M \abs{\nabla
  f}^p \rho_j^p + \abs{\nabla \rho_j}^p \abs f^p \de v_g \right).
  \end{equation}
  Summing up those inequalities proves as we claimed that
  \begin{equation}
    \norm{f}_{\CW^{1,p}(M,\mu)}^p  \le (1 + 2^{2p-1})
    K_{p,\mu} \sup_{j} \norm{\rho_j}^p_{\RC^1(M)} \norm{f}_{\RW^{1,p}(M)}^p.
  \end{equation}
\end{proof}
As an immediate corollary, we get the necessity in Theorem \ref{thm:charac}.

\begin{cor}
  \label{cor:pgedcompact}
  Let $p \in (1,2]$ and $\mu$ be a $p$-admissible measure on $M$. Then, the
  identity map on $\RC^\infty(M)$ extends to a compact operator $T_p^\mu :
  \RW^{1,p}(M) \to \RL^{p}(M,\mu)$.
\end{cor}

\begin{proof}
  Proposition \ref{prop:babysteps} implies that the identity map on $\RC^\infty(M)$
  extends to a bounded map $j : \RW^{1,p}(M) \to \CW^{1,p}(M,\mu)$.
  
  But then
  $T_p^\mu = \tau_p^\mu \circ j$ is an extension of the identity which is the
  composition of a compact and bounded operator, hence itself compact.
\end{proof}

One of our main tools going forward is estimates on (weak) solutions to the differential equation
  \begin{equation}
    \label{eq:weakeqn}
    \begin{cases}
      - \Delta \phi_{\xi,\mu} = \mu - \frac{\mu(M)}{\xi(M)} \xi & \text{in } M
      \\
      \del_\nu \phi_{\xi,\mu} = 0 &\text{on } \del M,
    \end{cases}
  \end{equation}
  for measures $\xi$ and $\mu$. Note that  $\mu -
  \frac{\mu(M)}{\xi(M)} \xi$ vanishes on constant fonction, if they are shown to
  be in $\RW^{1,p}(M)^*$ existence of a solution is easily guaranteed; we are specifically interested in
  estimating its norm in terms of trace operators and the Poincar\'e constants $K_{p}$.
  We require a generalisation of the Lax--Milgram theorem to Banach spaces.
  \begin{theorem}[{Banach--Ne\v cas--Babu\v ska Theorem,
    \cite[Theorem 2.6]{ernguermond}}]
    \label{thm:bnb}
    Let $X$ and $Y$ be real Banach spaces, with $Y$ being reflexive. Let $a$ be a
    bilinear form on $X \times Y$. Then, for every $L \in Y^*$ there is a unique
    $x \in X$ such that for all $y \in Y$,
    \begin{equation}
      a(x,y) = \langle L,y\rangle 
    \end{equation}
    if and only if $a$ satisfies the \emph{Brezzi condition}, i.e. there exists
    $\kappa > 0$ such that
    \begin{equation}
      \forall x \in X, \quad \kappa \norm{x}_X \le \sup_{y \in Y}
      \frac{a(x,y)}{\norm{y}_Y}
    \end{equation}
    and $a$ is weakly nondegenerate, i.e. if $a(x,y) = 0$ for all $x \in X$, then $y =
    0$.
  \end{theorem}
  Our goal is to use the Banach--Ne\v cas--Babu\v ska Theorem with $a :
  \CW^{1,p'}(M,\mu) \times \CW^{1,p}(M,\mu) \to \R$ given by  
\begin{equation}
  a(\phi,f) = \int_M \nabla \phi \cdot \nabla f \de v_g.
\end{equation}
It is clearly weakly nondegenerate if we restrict ourselves to functions of zero mean.
The following lemma establishes the Brezzi condition.
\begin{lemma}
  \label{lem:easypoincare}
  Let $M$ be a  Riemannian manifold, $p \in (1,\infty)$ with
   H\"older conjugate $p' = p/(p-1)$
   and $\mu$ a Radon measure supporting a $p$-Poincar\'e inequality and such that $\tau_{p'}^\mu$ is
   compact.
  Then, there exists $\kappa > 0$ such that for all $\phi \in \CW^{1,p'}(M,\mu)$,
\begin{equation}
  \label{eq:brezzi}
  \kappa \norm{\phi - m_{\phi,\mu}}_{\CW^{1,p'}(M,\mu)} \le 
  \sup_{0 \not \equiv f \in
  \CW^{1,p}(M,\mu)} \, \, \frac{\int_M \nabla \phi \cdot \nabla f \de
  v_g}{\norm{f}_{\CW^{1,p}(M,\mu)}}.
\end{equation}
\end{lemma}
\begin{proof}
  Towards a contradiction, we assume that such a $\kappa$ does not exist. This
  implies the existence of a sequence $\phi_n \in \CW^{1,p'}(M,\mu)$ such that
  \begin{equation}
    \norm{\phi_n - m_{\phi_n,\mu}}_{\RL^{p'}(M,\mu)}^{p'} + \norm{\nabla
    \phi_n}_{\RL^{p'}(M)}^{p'} = 1
  \end{equation}
  and 
  \begin{equation}
    \label{eq:coercivitycond}
  \sup_{0 \not \equiv f \in
  \CW^{1,p}(M,\mu)} \, \, \frac{\int_M \nabla \phi_n \cdot \nabla f \de
  v_g}{\norm{f}_{\CW^{1,p}(M,\mu)}} \xrightarrow{n \to \infty} 0.
  \end{equation}
  We first prove that if \eqref{eq:coercivitycond} goes to $0$, then $\abs{\nabla \phi_n}_{\RL^{p'}(M)}$ does as
  well. Since $\mu$ supports a $p$-Poincar\'e inequality, we have that
  \begin{equation}
    \label{eq:padminfsup}
    \begin{aligned}
  \sup_{0 \not \equiv f \in
  \CW^{1,p}(M,\mu)} \, \, \frac{\int_M \nabla \phi_n \cdot \nabla f \de
  v_g}{\norm{f}_{\CW^{1,p}(M,\mu)}} &\ge 
  \sup_{\substack{0 \not \equiv f \in
  \CW^{1,p}(M,\mu) \\ m_{f,\mu} = 0}} \, \, \frac{\int_M \nabla \phi_n \cdot \nabla f \de
  v_g}{\norm{f}_{\CW^{1,p}(M,\mu)}} \\
  &\ge 
  \sup_{\substack{0 \not \equiv f \in
  \CW^{1,p}(M,\mu)}} \, \, \frac{\int_M \nabla \phi_n \cdot \nabla f \de
  v_g}{(1 + K_{p,\mu})\norm{\nabla f}_{\RL^{p}(M)}}.
    \end{aligned}
  \end{equation}
  By density of smooth vector fields and duality, we have that
  \begin{equation}
    \label{eq:normduality}
    \norm{\nabla \phi}_{\RL^{p'}(M)} = \sup_{F \in \Gamma(TM)} \quad
    \frac{\int_M \nabla \phi\cdot  F \rd v_g}{\norm{F}_{\RL^p(M)}},
  \end{equation}
  where $\Gamma(TM)$ is the set of smooth vector fields on $M$.

  By the Helmholtz decomposition, of vector fields, see \cite[p.87]{schwarz},
  we can write $F = F_1 + F_2$ where $\div F_1 = 0$, $F_1 \cdot \nu
  \big|_{\del M} = 0$ and $F_2 = \nabla f$.
  Here $\nu$ is the
  normal vector to the boundary. By the divergence theorem
  \begin{equation}
    \int_{M} \nabla \phi \cdot F_1 \rd v_g = \int_M \div(\phi F_1) - \phi \div
    F_1 \de v_g = \int_{\del M} \phi F_1 \cdot \nu \de A_g = 0.
  \end{equation}
  Therefore, in \eqref{eq:normduality} we may assume that $F$ is a gradient
  field. Thus, from \eqref{eq:padminfsup}, we see that if
  \eqref{eq:coercivitycond} holds then $\abs{\nabla \phi_n}_{\RL^{p'}} \to 0$.
  Therefore, we have that $\phi_n$ is a sequence in $\CW^{1,p'}(M,\mu)$ so that
  \begin{equation}
    \norm{\phi_n - m_{\phi_n,\mu}}_{\RL^{p'}(M,\mu)} \xrightarrow{n\to\infty} 1 \qquad \text{and}
 \qquad \norm{\nabla \phi_n}_{\RL^{p'}(M)} \xrightarrow{n \to \infty} 0.
  \end{equation}
  By compactness of $\tau_{p'}^{\mu}$ there is $\phi \in \CW^{1,p}(M,\mu)$ such
that $\phi_n$ converges to $\phi$ weakly in $\CW^{1,p'}(M,\mu)$ and strongly in
$\RL^{p'}(M)$. In other words, $\phi$ is such that
\begin{equation}
  \norm{\phi - m_{\phi,\mu}}_{\RL^{p'}(M,\mu)} = 1 \quad \text{and} \quad
  \norm{\nabla \phi}_{\RL^{p'}(M)} = 0.
\end{equation}
This means that $\phi$ is constant a.e., and since $\tau_{p'}^{\mu}$ extends
the identity on $\RC^{\infty}(M)$, $\phi$ is also $\mu$-a.e. constant. But then,
$\norm{\phi -  m_{\phi,\mu}}_{\RL^{p'}(M,\mu)} = 0$, a contradiction.
\end{proof}

\begin{lemma}
  \label{lem:integralidentity}
  Let $M$ be a compact Riemannian manifold, $p \in (1,\infty)$, $\xi$ a Radon
  measure that supports a $p$-Poincar\'e inequality and such that
  $\tau_{p'}^\xi$ is compact; and $\mu$ be a Radon measure such that the identity on $\RC^\infty(M)$ extends to a bounded operator
  $T_p^{\xi,\mu} : \CW^{1,p}(M,\xi) \to \RL^p(M,\mu)$. Then, there exists a
  unique $\phi_{\xi,\mu} \in
  \CW^{1,p'}(M,\xi)$
  with $m_{\phi_{\xi,\mu},\xi} = 0$ and such that for all $f \in
  \CW^{1,p}(M,\xi)$
  \begin{equation}
    \label{eq:phiidentity}
    \int_M \nabla f \cdot \nabla \phi_{\xi,\mu} \de v_g = \int_M f \de \mu -
  \frac{\mu(M)}{\xi(M)} \int_M f \de \xi.
  \end{equation}
  Moreover, if $\mu$ supports a $p$-Poincar\'e inequality, $\phi_{\xi,\mu}$ satisfies
  \begin{equation}
    \label{eq:phiestimate}
    \norm{\nabla \phi_{\xi,\mu}}_{\RL^{p'}(M)}
    \le (1 + K_{p,\rd v_g})
    \mu(M)^{1/p'}
    \norm{T_p^\mu}.
  \end{equation}
\end{lemma}
\begin{remark}
  \label{rem:alreadydone}
  The condition on the existence of $T_p^{\xi,\mu}$ is later shown to always be
  satisfied for $p$-admissible measures. 
\end{remark}
\begin{proof}
  Let
\begin{equation}
  X_p := \set{f \in \CW^{1,p}(M,\xi) : m_{f,\xi} = 0} 
\end{equation}
and consider the bilinear form $a : X_{p'} \times X_p \to \R$ given by
\begin{equation}
  \label{eq:acoercive}
  a(\phi,f) = \int_M \nabla \phi \cdot \nabla f \de v_g.
\end{equation}
It follows from Lemma \ref{lem:easypoincare} that $a$ satisfies the Brezzi
condition, and it is weakly nondegenerate on $X_{p'} \times X_p$. 
Furthermore, since $\mu$ has finite volume $1 \in \RL^{p'}(M,\mu)$. This means that
$L:= (T_{p}^{\xi,\mu})^* 1 \in
X_p^*$ and for $f \in X_p$
\begin{equation}
  \label{eq:ldef}
  \langle L,f \rangle := \int_M f \de \mu
\end{equation}
By the
Banach--Ne\v cas--Babu\v ska theorem there exists a unique $\phi_{\mu,\xi} \in
X_{p'}$ so that for all $f \in X_p$, $a(\phi_{\xi,\mu},f) = L(f)$. For $f \in
\CW^{1,p}(M,\xi)$, we obtain the identity
\eqref{eq:phiidentity} by noticing that formula \eqref{eq:ldef} extends from
$X_p$ to $\CW^{1,p}(M,\xi)$ and computing
\begin{equation}
 \begin{aligned}
 \langle L,f \rangle = \langle L,f - m_{f,\xi}\rangle  + \langle L,
 m_{f,\xi}\rangle  = a(\phi_{\mu,\xi},f) +
   \frac{\mu(M)}{\xi(M)} \int_{M} f \de \xi.
 \end{aligned}
\end{equation}

We turn our attention to estimate \eqref{eq:phiestimate}. As in the proof of
Lemma \ref{lem:easypoincare}, we have that
\begin{equation}
  \label{eq:duality}
  \norm{\nabla \phi_{\xi,\mu}}_{\RL^{p'}(M)} = \sup_{f \in \RC^\infty(M)}
  \frac{\int_M \nabla \phi_{\xi,\mu} \cdot \nabla f \de v_g}{\norm{\nabla
  f}_{\RL^{p}(M)}}.
\end{equation}
From the weak characterisation of
$\phi_{\xi,\mu}$ that for any $f \in \RC^\infty(M)$,
\begin{equation}
  \label{eq:gradientfield}
  \begin{aligned}
    \abs{\int_M \nabla f \cdot \nabla
    \phi_{\xi,\mu} \de v_g} = \abs{\int_M f \de \mu - \frac{\mu(M)}{\xi(M)} \int_M f
    \de \xi}.
  \end{aligned}
\end{equation}
Since the lefthand side is invariant under addition of a constant to $f$, we may
assume that $\int_M f \de \xi = 0$. By H\"older's inequality, if $\mu$ supports
a $p$-Poincar\'e inequality we have that
\begin{equation}
  \begin{aligned}
    \abs{\int_M f \de \mu} \le \mu(M)^{1/p'} \norm{f}_{\RL^{p}(M,\mu)} &
    &\le \mu(M)^{p'}(1 + K_{p,\rd v_g}) \norm{T_p^\mu}
  \norm{\nabla f}_{\RL^p(M)},
\end{aligned}
\end{equation}
where $T_p^\mu$ is bounded from Proposition \ref{prop:babysteps}. Inserting this estimate into \eqref{eq:gradientfield} and \eqref{eq:duality} completes the
proof.
\end{proof}
We can now prove that the spaces $\RW^{1,p}(M)$ and $\CW^{1,p}(M,\mu)$ are
isomorphic.
\begin{prop}
  \label{prop:secndbnd}
  Suppose that the identity on $\RC^\infty(M)$ extends to a bounded operator
  $T_p^\mu : \RW^{1,p}(M) \to \RL^p(M,\mu)$. Then, there is
  $C_{p,\mu}$ such that
  \begin{equation}
    \label{eq:secndbound}
    \norm{f}_{\RW^{1,p}(M)} \le C_{p,\mu} \norm{f}_{\CW^{1,p}(M,\mu)}.
  \end{equation}
  If moreover $\mu$ supports a $p$-Poincar\'e inequality, then we can take
  \begin{equation}
    C_{p,\mu} = \left( 1 + K_{p,\rd v_g} \right)\left( 1 + \frac{\Vol_g(M)^{1 +
    \frac 1 p}}{\mu(M)^{1 - \frac{1}{p'}}}\right)\left(1 + \norm{T_p^\mu}
          \right).
  \end{equation}
\end{prop}
Before carrying on with the proof, we note that we have proved in Proposition
\ref{prop:babysteps} that supporting a $p$-Poincar\'e inequality implies that
$T_p^\mu$ is bounded, so that this proposition implies the second bound in
Theorem \ref{thm:normeq}.
\begin{proof}
We have that
\begin{equation}
  \label{eq:boundsplt}
  \begin{aligned}
    \norm{f}_{\RL^p(M)} &\le  \norm{f - m_{f}}_{\RL^p(M)}+
    \norm{m_{f}}_{\RL^p(M)} \\
    &\le  K_{p,\rd v_g}^{1/p} \norm{\nabla f}_{\RL^p(M)} + 
    \Vol_g(M)^{1/p} \abs{m_{f}}.
\end{aligned}
\end{equation}
From Lemma
\ref{lem:integralidentity} with $\xi = \rd v_g$, there is $\phi \in \RW^{1,p'}(M)$ such that
\begin{equation}
  \label{eq:boundavg}
  \begin{aligned}
    \Vol_g(M)^{1/p} \abs{m_{f}} &\le  \frac{\Vol_g(M)^{1+\frac 1 p}}{\mu(M)} \left[
      \abs{ \int_M
        \nabla f \cdot \nabla \phi \de v_g
    } + \abs{ \int_M f \de \mu } \right] \\
    &\le \frac{\Vol_g(M)^{1+1/p}}{\mu(M)^{1 - \frac{1}{p'}}}\left[ \norm{\nabla
    \phi}_{\RL^{p'}(M)} \norm{\nabla f}_{\RL^p(M)} + \norm{f}_{\RL^p(M,\mu)} \right].
\end{aligned}
\end{equation}
The estimate on $C_{p,\mu}$ can be then be read from the bound on $\norm{\nabla
\phi}_{\RL^{p'}(M)}$ obtained in Lemma \ref{lem:integralidentity} under the
$p$-Poincar\'e inequality condition.
\end{proof}

We can now prove sufficience in Theorem \ref{thm:charac}.

\begin{prop}
  Let $\mu$ be a Radon measure on $M$ and suppose that the identity map on
  $\RC^\infty(M)$ extends to a compact operator $T_p^\mu : \RW^{1,p}(M) \to
  \RL^p(M,\mu)$. Then, $\mu$ is $p$-admissible.
\end{prop}
\begin{proof}
  Proposition \ref{prop:secndbnd} implies that the identity map on $\RC^\infty(M)$
  extends to a bounded map $j\colon\CW^{1,p}(M,\mu)\to\RW^{1,p}(M)$
Thus, since it is the composition of a compact and a bounded operator,
$\tau_p^\mu =
T_p^{\mu}\circ j$
is also compact. 
We now prove that the Poincar\'e inequality holds. Assume otherwise, then there exists a sequence of smooth functions $f_n$
such that
$$
\int_M f_n^p\de\mu = 1,\quad \int_M|\nabla f_n|^p\de v_g \to 0,\quad \int _M f_n\de\mu = 0.
$$
By Proposition~\ref{prop:secndbnd}, the functions $f_n$ are uniformly bounded in
$\RW^{1,p}(M)$. Since $T_{\mu}$ is compact, there is $f \in \RW^{1,p}(M)$ such
that, up to choosing a subsequence, $f_n \rightharpoonup f$  weakly in
$\RW^{1,p}(M)$ and $f_n \to T_p^\mu f$ strongly
in $\RL^p(M,\mu)$. By lower semicontinuity, $\norm{\nabla f}_{\RL^p(M)} = 0$,
therefore, $f$ is $\rd v_g$-a.e. constant. Since $T_p^{\mu}$ extends the identity
on $\RC^\infty(M)$, $T_p^\mu f$ is a $\mu$-a.e. constant, which
contradicts the fact that
$$
\int_M (T_{\mu}f)^p\de\mu = 1,\qquad \int _M T_{\mu}f\de\mu = 0.
$$

\end{proof}
We finally write the two following propositions that allow us to rewrite Lemma
\ref{lem:integralidentity} with weaker conditions, for ease of reference. The
first proposition indicates that $p$-admissibility is a monotone condition.

\begin{prop}
  Suppose that $T_p^\mu : \RW^{1,p}(M) \to \RL^p(M,\mu)$ is bounded. Then for
  all $q > p$, $T_q^\mu$ is compact. In particular, $\mu$ is
  $q$-admissible and admissibility is a monotone condition.
\end{prop}

\begin{proof}
  If $q > d$, $T_q^\mu$ is compact since the embedding $\RW^{1,q}(M) \to
  \RC(M)$ is compact, so we suppose now that $p < q \le d$. Compactness of $T_q^\mu$ follows from
  general
  interpolation theory. Given two compatible normed vector spaces, i.e. spaces
  $X_0, X_1$ that are both subspaces
  of a larger topological vector space $V$, Peetre's
  $K$-functional is defined on $f \in X_0 + X_1$ as 
  \begin{equation}
    K(f,t,X_0,X_1) := \inf \set{\norm{f_0}_{X_0} + t \norm{f_1}_{X_1} :
    f = f_0 + f_1, \, f_0 \in X_0, \,  f_1 \in X_1}.
  \end{equation}
  For $ 0 < \theta < 1$ and $1 \le q < \infty$, let
$(X_0,X_1)_{\theta,q}$ be the interpolation space between $X_0$ and $X_1$ (see
\cite[Chapter 5]{benettsharpley}):
\begin{equation}
  (X_0,X_1)_{\theta,q} := \set{f \in X_0 + X_1 : \norm{f}_{\theta,q} :=
    \left( \int_0^\infty \left( t^{-\theta} K(f,t,X_0,X_1) \right)^q
  \frac{\rd t}{t} \right)^{\frac 1 q} < \infty}.
\end{equation}

We use the interpolation theorem found in \cite[Theorem 2.1]{cobosfernandez}
which states the following. Given $Y_0, Y_1$
compatible Banach spaces; $X_0, X_1$ Banach spaces such that $X_1$ is
continuously embedded in $X_0$ and $T$ is a linear operator
such that $T : X_0 \to Y_0$ is bounded and the restriction $T : X_1 \to Y_1$ is
compact. Then, for $0 < \theta < 1$ and $1 \le q < \infty$, the operator
\begin{equation}
  T : (X_0,X_1)_{\theta,q} \to (Y_0,Y_1)_{\theta,q} \qquad \text{is compact.}
\end{equation}
 
Let $r > d$ and $\theta = \frac{r(q-p)}{q(r-p)} < 1$. We take $Y_0$ to be
$\RL^p(M,\mu)$ and $Y_1$ to be $\RL^{r}(M,\mu)$, it follows from \cite[Theorems
5.1.9 and 5.2.4]{benettsharpley} that $(Y_0,Y_1)_{\theta,q} = \RL^q(M,\mu)$. 

On the other hand, taking $X_0 = \RW^{1,p}(M)$ and $X_1 = \RW^{1,r}(M)$ it
follows from \cite[Theorem 6.2, Corollary 1.3 and Remark 4.3]{badr}, the later
remark treating the case with Lipschitz boundary, that
\begin{equation}
  (X_0,X_1)_{\theta,q} = \RW^{1,q}(M).
\end{equation}
Therefore, the interpolation theorem tells us that $T_q^\mu : \RW^{1,q}(M) \to
\RL^q(M,\mu)$ is indeed compact. 

\end{proof}

\begin{prop}
  Let $\mu,\xi$ be two admissible measures. Then, the identity on
  $\RC^\infty(M)$ extends to a compact operator $T_p^{\xi,\mu} :
  \CW^{1,p}(M,\xi) \to \RL^p(M,\mu)$.
\end{prop}
 
\begin{proof}
Define $T_p^{\xi,\mu}$ as the composition
\begin{equation}
  \begin{tikzcd}
    \CW^{1,p}(M,\xi) \arrow{r}{j} \arrow[swap]{dr}{T_p^{\xi,\mu}} &\RW^{1,p}(M)
    \arrow{d}{T_\mu} \\
& \RL^p(M,\mu)
  \end{tikzcd}
\end{equation}
where by Theorem \ref{thm:normeq} $j$ is bounded since $\xi$ is $p$-admissible.
By Theorem \ref{thm:charac},  $T_p^\mu$ is compact. Thus, $T_p^{\xi,\mu}$ is
compact as the composition of a compact and a bounded operator, and it is an
extension of the identity on $\RC^\infty(M)$.
\end{proof}
We therefore rewrite the statement of Lemma \ref{lem:integralidentity} in
the following way.
\begin{lemma}
  \label{lem:integralidentitybis}
  Let $M$ be a compact Riemannian manifold, $p \in (1,2]$, $\xi$ and $\mu$ both
  $p$-admissible measures. Then, there exists a unique $\phi_{\xi,\mu} \in
  \RW^{1,p'}(M)$ with $m_{\phi_{\xi,\mu},\xi} = 0$ and such that for all $f \in
  \RW^{1,p}(M)$,
  \begin{equation}
    \int_M \nabla f \cdot \nabla \phi_{\xi,\mu} \de v_g = \int_M f \de \mu -
    \frac{\mu(M)}{\xi(M)} \int_M f \de \xi. 
  \end{equation}
  Moreover, 
  \begin{equation}
    \norm{\nabla \phi_{\xi,\mu}}_{\RL^{p'}(M)} \le (1 + K_{p,\rd
    v_g})\mu(M)^{1/p'} \norm{T_p^\mu}. 
  \end{equation}
\end{lemma}
\subsection{Examples and admissibility criteria}
\label{sec:examples}

Let us now give a few examples of admissible measures, as well as a
local criterion that characterises them. We start with basic examples.
  
\begin{ex}
  On a smooth compact manifold $M$ with Lipschitz boundary, the volume measure $\rd v_g$ is
  $p$-admissible for every $p \in (1,\infty)$, as is the pushforward by
  inclusion of the boundary measure $\iota_* \de A_g$. Any linear combination of
  them is also $p$-admissible.
\end{ex}

\begin{ex}
  It follows
  from the definition of the capacity that measures supported on a set of
  $p$-capacity zero do not support a
  $p$-Poincar\'e inequality, and as such are not admissible.
\end{ex}

We now explore the edge cases
of admissibility. We provide those examples for $p = 2$ since that is the
context where they will be relevant. This last example allows us to obtain the
weakest integrability condition on $\beta$ in Theorem \ref{thm:homointro}.
We need to introduce the following characterisation of compactness beforehand.
  Maz'ya's compactness
  criterion \cite[Section 11.9.1]{mazyasobolevspaces} states that $T_p^\mu$ is
  compact if and only if
  \begin{equation}
    \label{eq:mazyacriterion}
    \lim_{r \to 0^+} \sup \set{\frac{\mu(\Upsilon)}{\capa_p(\Upsilon)} : \Upsilon \subset
    M, \, \, \operatorname{diam}(\Upsilon) \le r} = 0.
  \end{equation}
 The isocapacitary inequality \cite[Equations
2.2.11 and 2.2.12]{mazyasobolevspaces} states that for every $\Upsilon \subset M$, with
  $\Vol_g(\Upsilon) \le \Vol_g(M)/2$ that
  \begin{equation}
    \label{eq:isocapacitary}
    \capa_p(\Upsilon) \gg_M \begin{cases}
      \log(1/\Vol_g(\Upsilon))^{1-d} & \text{if } p = d, \\
      \Vol_g(\Upsilon)^{\frac{d-p}{d}} & \text{if } d > p.
    \end{cases}
  \end{equation}
\begin{ex}
  \label{ex:lpdensity}
  Let $0 \le \beta \in \LLL(M)$ (for $d = 2$) or $0 \le \beta \in \RL^{d/2}(M)$ (for $d > 2$)
  be a positive density and $\mu = \beta \rd v_g$. Then, $\mu$ is admissible.
  For $1 \le p < d/2$ (for $d\geq 3$), and for $p = 1$ (when $d = 2$), there exists $\beta \in
  \RL^p(M)$ such that $\beta \rd v_g$ is not an admissible measure. We split the
  proof of these claims in a few cases.

  \textbf{Case (i): $\mathbf{ p < d/2}$, $\mathbf{d \ge 3}$.}
  Consider
  any $x \in M$, $r_y = \dist(x,y)$ and $$\beta(y) = \max\set{\frac{1}{r_y^{d/p}
  \log(1/r_y)},1}.$$ It is easy to see that $\beta \in \RL^p$ and that the
  measure $\beta \rd v_g$ fails Maz'ya's compactness criterion, hence
  $T_\mu$ is not compact and $\mu$ is not admissible. 

  \textbf{Case (ii): $\mathbf{p = 1}$, $\mathbf{d = 2}$.} 
  Similarly to the previous case, for some $0 < \delta < 1$ let
  \begin{equation}
    \beta(y) = \max\set{\frac{1}{r_y^2 \log(1/r_y)^{1 + \delta},1}}.
  \end{equation}
  This time, $\beta \in \RL^1(M)$ but we can see that $T_\mu$ is not even
  bounded on $\RW^{1,2}(M)$, so certainly not admissible. Indeed, consider the
  function $f(y) = - \log(r_y)^a$. It is a standard computation to see that $f
  \in \RW^{1,2}(M)$ when $a < 1/2$. On the other hand, choosing $a = \delta/2$, for
  some $\eps > 0$
  \begin{equation}
    \int_M f^2 \beta \de v_g \gg_M \int_{0}^\eps \frac{\de r}{r \log(1/r)} =
    +\infty.
  \end{equation}

  \textbf{Case (iii): $\mathbf{d=2}$ and $\boldsymbol{\beta \in \LLL(M)}$ or
  $\mathbf{d \ge 3}$ and $\mathbf{p \ge d/2}$.}
  Suppose without loss of generality that $\abs \beta \ge 1$ a.e. 
  For $d = 2$, it follows from Jensen's inequality with the convex function
  $\phi(x) = x \log x$ that for any $\Upsilon \subset M$, 
  \begin{equation}
 \log\left(\frac{1}{\Vol_g(\Upsilon)}
    \int_{\Upsilon} \beta \de v_g\right)\int_\Upsilon \beta \de v_g \le
    \int_{\Upsilon} \beta \log \beta \de
    v_g.
  \end{equation}
  In other words, 
\begin{equation}
  \log\left( \frac{1}{\Vol_g(\Upsilon)} \right)^{-1} \ge
  \frac{\mu(\Upsilon)}{\norm{\beta}_{\LLL(\Upsilon)} -
  \mu(\Upsilon)\log(\mu(\Upsilon))}.
\end{equation}
Here, we supposed that $\Upsilon$ is chosen with diameter small enough that
$\log(\mu(\Upsilon)) < 0$, ensuring that all quantities involved are positive. 
Inserting into Maz'ya's compactness criterion \eqref{eq:mazyacriterion} along with the isocapacitary
inequality \eqref{eq:isocapacitary} gives
\begin{equation}
  \lim_{r \to 0^+} \sup\set{\frac{\mu(\Upsilon)}{\capa_2(\Upsilon)} : \diam(\Upsilon) \le r} \le \lim_{r
  \to 0^+} \sup \set{\norm{\beta}_{\LLL(\Upsilon)} - \mu(\Upsilon)
  \log(\mu(\Upsilon)) : \diam(\Upsilon) \le r}.
\end{equation}
Note that if $\norm{\beta}_{\LLL(\Upsilon)}$ goes to $0$ uniformly in $\Upsilon$,
then so does $\mu(\Upsilon)$, hence $\mu(\Upsilon) \log(\mu(\Upsilon))$ as well. 

For every $\Upsilon$, let $\Upsilon_m = \Upsilon \cap \set{\beta \log \beta < m}$,
and observe that $\norm{\beta}_{\LLL(\Upsilon)} = \norm{\beta}_{\LLL(\Upsilon_m)} +
\norm \beta_{\LLL(\Upsilon \setminus \Upsilon_m)}$.  
It follows from density of smooth functions in $\RL^1$ that for every $\eps >
0$, there is $m_\eps$ large enough so that for every $m > m_\eps$,
\begin{equation}
  \norm \beta_{\LLL(\Upsilon \setminus \Upsilon_m)} \le \norm\beta_{\LLL(M
  \setminus M_m)} \le \eps,
\end{equation}
which hence holds uniformly in $\Upsilon$. On the other hand, if $\diam \Upsilon
\le r$,
\begin{equation}
  \norm{\beta}_{\LLL(\Upsilon_m)} \le m \log m r^2.
\end{equation}
Taking $m = r^{-1}$, we have that for $r$ small enough $m \ge m_\eps$ and we
deduce that for every $\eps > 0$,
\begin{equation}
  \lim_{r \to 0^+} \sup\set{\norm \beta_{\LLL(\Upsilon)} : \diam \Upsilon \le r}
  \le \eps
\end{equation}
so that by Maz'ya's compactness criterion $T_\mu$ is compact and $\mu$ is
admissible.

For $d > 2$, it follows from H\"older's inequality that 
\begin{equation}
  \Vol_g(\Upsilon)^{\frac{2-d}{d}} \le \frac{\norm{\beta}_{\RL^{d/2}(\Upsilon)}}{\mu(\Upsilon)}.
\end{equation}
Therefore, inserting in Maz'ya's criterion along with the isocapacitary inequality
\begin{equation}
  \lim_{r \to 0^+} \sup \set{\frac{\mu(\Upsilon)}{\capa_2(\Upsilon)} :
  \diam \Upsilon \le r} \le 
  \lim_{r \to 0^+} \sup \set{\norm{\beta}_{\RL^{d/2}(\Upsilon)} : \diam \Upsilon
  \le r}.
\end{equation}
The same argument as earlier but with the sets $\Upsilon_m = \Upsilon \cap
\set{\beta^{d/2} < m}$ shows that this limit converges to $0$, so that $\mu$ is
admissible.

\end{ex}

\section{Variational eigenvalues}
\label{variationalev:sec}

Eigenvalue convergence results are ubiquitous in the literature and the proofs
of a large number of them follow similar steps. In
the present section we formulate these steps explicitly in sufficient generality
to allow direct application to many natural eigenvalue problems,
including both the Steklov and Laplace problems.

\subsection{Variational eigenvalues associated to a Radon measure}
We generalise to higher dimension the definition of eigenvalues associated to a
measure, introduced in~\cite{kok} for surfaces.
Let $(M,g)$ be a compact Riemannian manifold.

For a Radon measure $\mu$
 on $M$, we define the variational eigenvalues
$\lambda_k(M,g,\mu)$ in the following way. For any $f \in \RC^\infty(M)$ such
that $f \not \equiv 0$  in $\RL^2(M,\mu)$, we define the Rayleigh quotient $R_g(f,\mu)$  by
$$
R_g(f,\mu) := \frac{\displaystyle\int_M|\nabla f|^2_g \de
v_g}{\displaystyle\int_M f^2\de \mu}. 
$$
The eigenvalues $\lambda_k(M,g,\mu)$ are then given by
\begin{equation}
\label{MeasureLaplace:def}
\lambda_k(M,g,\mu) := 
\adjustlimits\inf_{F_{k+1}}\sup_{f\in
F_{k+1}\setminus\{0\}} R_g(f,\mu),
\end{equation}
where the infimum is taken over all $(k+1)$-dimensional subspaces
$F_{k+1}\subset \RC^\infty(M)$ that remain $(k+1)$-dimensional in
$\RL^2(M,\mu)$. A natural normalisation for these eigenvalues is
\begin{equation}
  \normallap{k}(M,g,\mu) := \lambda_k(M,g,\mu)\frac{\mu(M)}{\Vol_g(M)^{\frac{d-2}{d}}},
\end{equation}
see e.g. \cite{GNY2004}.

The following proposition states that the eigenvalues of
admissible measures possess all the natural properties one expects from
eigenvalues of an operator of Laplace-type. 
\begin{prop}
\label{ef:prop}
Let $\mu$ be an admissible measure.
Then one has 
$$
0=\lambda_0(M,g,\mu)<\lambda_1(M,g,\mu)\leqslant
\lambda_2(M,g,\mu)\leqslant\ldots\nearrow\infty;
$$
i.e. the first eigenvalue is positive, the multiplicity of each eigenvalue is
finite, and the eigenvalues tend to $+\infty$. Moreover, there exists an
orthogonal basis of eigenfunctions $f_j\in \CW^{1,2}(M,\mu)$ satisfying
\begin{equation}
\label{eigenfunctions_measures:def}
\int_M  \nabla f_j \cdot \nabla u \de v_g = \lambda_j(M,g,\mu)\int_M f_j u\de \mu
\end{equation}
for all $u\in \CW^{1,2}(M,\mu)$.
\end{prop}
\begin{proof}
  That $\lambda_1(M,g,\mu) > 0$ is readily seen to be equivalent to $\mu$
  supporting a $2$-Poincar\'e inequality. 
  The rest of the proof is standard. The bilinear form
  $$
  a(f,\phi) = \int_M \nabla f \cdot \nabla \phi \de v_g
  $$
  is bounded and coercive on the set of functions of $\mu$-average $0$ in $\CW^{1,2}(M,\mu)$. The statement is therefore a direct application of
  \cite[Theorem 6.3.4]{blanchardbruning} and the Courant--Fischer--Weyl minmax
  principle.
\end{proof}

We revisit the examples from the previous section and how they give rise to
natural eigenvalues.

  \begin{ex}
    If $M$ is closed and $\mu=\mathrm dv_g$, the volume measure associated to $g$,
    then $\lambda_k(M,g,\mathrm d v_g)$ are eigenvalue of the Laplace operator. In this case $T_\mu$ is the usual embedding
    $\RW^{1,2}(M)\subset \RL^2(M)$.
    If $M$ is a compact manifold with boundary, then $\lambda_k(M,g,\mathrm
    dv_g)$ are Neumann eigenvalues. 
  \end{ex}

  \begin{ex}
    If $M$ is a compact manifold with boundary and $\mu=\iota_*\mathrm d A_g$, the
    pushforward by inclusion of the induced volume measure on $\partial M$, then
    $\lambda_k(M,g,\mu)$ are Steklov eigenvalues.
\end{ex}

\begin{ex}
\label{ex:transmission}:
  If $M$ is a compact manifold, $\Sigma \subset M$ is a closed smooth
 hypersurface in the interior of $M$, and $\mu = \iota_* \rd A_g^\Sigma$ is the
 pushforward by inclusion of the induced volume measure on $\Sigma$, then
 $\lambda_k(M,g,\mu)$ are the eigenvalues of the transmission problem
 \begin{equation}
   \begin{cases}
     \Delta u = 0 & \text{in } M \setminus \Sigma, \\
     (\del_{n^+} + \del_{n^-}) u = \lambda u & \text{on } \Sigma.
   \end{cases}
 \end{equation}
 where $\del_{n^\pm}$ are normal derivatives in opposite directions on
 $\Sigma$. 
\end{ex}

\begin{ex}
  For $\beta > 0$, if $\mu = \iota_* \mathrm dA_g + \beta\de v_g$, then $\lambda_k(M,g,\mu)$ are
  eigenvalues associated with a dynamical boundary value problem, see
  \cite{belowfrancois,GHL}, given by
  \begin{equation}
    \begin{cases}
  -\Delta f = \lambda \beta f & \text{in } M,\\
  \del_nf = \lambda f & \text{on } \del M.
\end{cases}
  \end{equation}
  The corresponding Laplace-type operator acts in
  $\RL^2(M) \oplus \RL^2(\partial M, \rd A_g)$ and is not densely defined. From
  the perspective of variational eigenvalues, this does not cause any problem.
\end{ex}

\begin{ex}
  For $0 \le \beta \in \RL^{d/2}(M)$ ($d=3$), or $0 \le \beta \in \LLL(M)$
  ($d=2$) and $\mu = \beta \de v_g$, then $\lambda_k(M,g,\mu)$ are the eigenvalues of the weighted
  problem
  \begin{equation}
    \begin{cases}
      -\Delta f =  \lambda \beta f & \text{in } M, \\
      \del_n f = 0 & \text{on } \del M.
    \end{cases}
  \end{equation}
  By Example \ref{ex:lpdensity} $\mu$ is admissible, so that the spectrum is
  indeed discrete.
\end{ex}

\subsection{Continuity of eigenvalues}
While the eigenvalues $\lambda_k(M,g,\mu)$ may not necessarily be continuous under
weak-$*$ convergence of measures, they are always upper-semicontinuous, see
\cite[Proposition 1.1]{kok} for $d=2$. We include the proof in this context for
completeness, but it is the same in essence.
\begin{prop}
\label{usc:prop}
Let $(M,g)$ be a Riemannian manifold and assume $\mu_n\wks\mu$. Then
$$
\limsup_{n\to\infty}\lambda_k(M,g,\mu_n)\leqslant \lambda_k(M,g,\mu)
$$
\end{prop}
\begin{proof}
Let $\varepsilon>0$ be arbitrary. Let $F\subset \RC^\infty(M)$ be a 
$(k+1)$-dimensional subspace that remains $(k+1)$-dimensional in $\RL^2(M,\mu)$
and such that
$$
\sup_{f\in
F\setminus\{0\}}R_g(f,\mu)\leqslant \lambda_k(M,g,\mu) + \varepsilon.
$$
Convergence $\mu_n\wks\mu$ implies that for large $n$ the subspace $F$ is
$(k+1)$-dimensional in $\RL^2(M,\mu_n)$ and
$$
\adjustlimits \lim_{n\to\infty}\sup_{f\in
F\setminus\{0\}}R_g(f,\mu_n) = \sup_{f\in
F\setminus\{0\}}R_g(f,\mu).
$$
As a result, for large $n$ one has
$$
\lambda_k(M,g,\mu_n)\leqslant \sup_{f\in
F\setminus\{0\}}R_g(f,\mu_n)\leqslant \lambda_k(M,g,\mu) + 2\varepsilon.
$$
\end{proof}

For many applications it is important to establish continuity of
eigenvalues.
To the best of the authors' knowledge there is no sufficiently general
condition that guarantees continuity of $\lambda_k(M,g,\mu)$ which can be
verified in an efficient manner in our current setting. As an
example, we note that all examples of convergence covered in the present paper
fail the integral
distance convergence criterion given in \cite[Section 4.2]{kok}. Many stronger
convergence criteria exists, see e.g.
\cite{AnnePostJST,burenkovlambertilanzadec,burenkovlamberti}, however they generally require explicit
knowledge of some transition operators between Hilbert spaces, which usually
means having explicit information about the eigenfunctions. Our goal is to
obtain synthetic criteria for eigenvalue and eigenfunction convergence which
depends only on the measures $\mu_n$, and potentially on domains $\Omega_n$ on
which it is supported. 

Let $\Omega_n \subset M$ be a sequence of domains viewed as Riemannian manifolds
with the metric induced on $M$, and $\set{\mu_n: n \in \N}$, $\mu$ be Radon measures so that
$\operatorname{supp}(\mu_n) \subset \overline{\Omega_n}$. We use the same
notation $g, \mu_n$ for their restrictions to $\Omega_n$. Suppose that
\begin{itemize}
  \item[\ci] $\mu_n\wks\mu$ and $\Vol_g(M \setminus \Omega_n) \to 0$;
  \item[\cii] the measures $\mu$, $\mu_n$ are admissible for all $n$;
  \item[\cvi] there is an equibounded family of extension maps $J_n : \CW^{1,2}(\Omega_n,\mu_n)
    \to \CW^{1,2}(M,\mu_n)$.
\end{itemize}
The condition \cii guarantees the existence of the $\mu_n$-orthonormal
collection of eigenfunctions $f_j^n \in \CW^{1,2}(\Omega_n,\mu_n)$ associated with $\lambda_j(\Omega_n,g,\mu_n)$.
In any situation where $\Omega_n = M$ for all $n$, the third condition and the
volume part of the first condition are automatically satisfied. The map $J_n$ is
often built using harmonic extensions, and the collection $J_n f_j^n$ remains
$\mu_n$ orthonormal.

We now describe two conditions for the eigenfunctions.
\begin{itemize}
  \item[\civ]  For all $u \in \CW^{1,2}(M,\mu)$, the functions $f_j^n$ satisfy
    \begin{equation}
      \lim_{n \to \infty} \abs{ \langle J_n f_j^n,u \rangle_{\RL^2(M,\mu)} -
        \langle f_j^n,u
      \rangle_{\RL^2(\Omega_n,\mu_n)}} = 0.
    \end{equation}
  \item[\cv]  For every $j, k \in \N$, the functions $f_j^n$, $f_k^n$ satisfy 
    \begin{equation}
      \lim_{n \to \infty}\abs{\langle J_n f_j^n, J_n f_k^n\rangle_{\RL^2(M,\mu)} -
      \langle f_j^n,f_k^n\rangle_{\RL^2(\Omega_n,\mu_n)}}
      = \lim_{n \to \infty}\abs{\langle f_j^n,f_k^n\rangle_{\RL^2(M,\mu)} -
      \delta_{jk}}
      = 0,
    \end{equation}
    where $\delta_{jk}$ is the Kronecker delta.
\end{itemize}

Condition \cv implies that $\set{f_j^n : n \in \N}$ is bounded in
$\CW^{1,2}(M,\mu)$, so that up to a subsequence, $f^n_j\rightharpoonup f_j$
weakly in $\CW^{1,2}(M,\mu)$ and
    $\lambda_j(\Omega_n,g,\mu_n)\to \lambda_j$ for some $\lambda_j \ge 0$.

    Condition \civ implies that the functions $f_j$ are eigenfunctions associated with $(M,g,\mu)$ with
    the corresponding eigenvalues $\lambda_j$.
At this point it is unclear whether $\lambda_j$ is indeed the $j$-th eigenvalue
$\lambda_j(M,g,\mu)$. This will follow from condition \cv, which says
essentially that the eigenfunctions do not lose mass in the limit.
We formalize this procedure in the following
proposition.
\begin{prop}
\label{stability1:prop}
Assume that the domains $\Omega_n \subset M$ and the Radon measures  $\mu_n$, $\mu$ on $(M,g)$ satisfy conditions
\ci--\cvi, and that the eigenfunctions associated with $\mu_n$ satisfy
conditions \ciii--\cv. Then 
$$
\lim_{n\to\infty} \lambda_j(\Omega_n,g,\mu_n) = \lambda_j(M,g,\mu),
$$
and, up to a choice of subsequence, 
\begin{equation}
  \lim_{n \to \infty} J_n f_j^n = f_j,
\end{equation}
weakly in $\CW^{1,2}(M,\mu)$. If $\lambda_j(M,g,\mu)$ is simple, the convergence is
along the whole sequence. Finally, if
\begin{equation}
  \label{eq:smallenergy}
  \lim_{n \to \infty} \norm{\nabla J_n f_j^n}_{\RL^{2}(M \setminus \Omega_n,\mu_n)} = 0,
\end{equation}
the convergence is strong in $\CW^{1,2}(M,\mu)$.
\end{prop}
\begin{proof}
  From the definition via Rayleigh quotient, we see that for all $j$ and $n$,
  $\lambda_j(\Omega_n,g,\mu_n) \le \lambda_j(M,g,\mu_n)$. 
By Proposition~\ref{usc:prop} along with Condition \ci, we have that up to a subsequence
$\lambda_j(\Omega_n,g,\mu_n) \to \lambda_j \le \lambda_j(M,g,\mu)$. 

For each fixed $j$
\begin{equation}
  \label{eq:normef}
\norm{J_n f_j^n}_{\CW^{1,2}(M,\mu)}^2 = \norm{f_j^n}_{\RL^2(M,\mu)}^2 + \norm{\nabla J_n
f_j^n}^2_{\RL^{2}(M \setminus \Omega_n,\mu)} + 
\lambda_j(\Omega_n,g,\mu_n).
\end{equation}
Therefore in view of conditions \cvi and \cv the sequence $\set{f_j^n : n \in \N}$ is
bounded in $\CW^{1,2}(M,\mu)$ so that up to a subsequence, $J_n f_j^n \wk f_j$
weakly in $\CW^{1,2}(M,\mu)$.

We now claim that $f_j$ is an eigenfunction associated with
$(M,g,\mu)$ and corresponding eigenvalue $\lambda_j$. Since all relevant
quantities are equibounded in $\CW^{1,2}(M,\mu)$, we may use smooth
functions as trial functions for $f_j$ and $\lambda_j$. By weak
convergence we have that for any $u \in \RC^\infty(M)$,
\begin{equation}
  \label{eq:energyconv}
  \lambda_j(\Omega_n,g,\mu_n) + \int_{M \setminus \Omega_n} \nabla J_n f_j^n
  \cdot \nabla u \de v_g =  \int_M \nabla J_n f_j^n \cdot \nabla u \de v_g  \xrightarrow{n \to \infty} \int_M
  \nabla f_j \cdot \nabla v \de v_g,
\end{equation}
and by conditions \ci and \cvi
\begin{equation}
  \label{eq:smallext}
  \int_{M \setminus \Omega_n} \nabla J_n f_j^n \cdot \nabla u \de v_g \le
  \Vol_g(M \setminus \Omega_n)^{1/2} \norm{J_n} \norm{u}_{\RC^1(M)}
  \norm{f_j^n}_{\CW^{1,2}(M,\mu_n)} \xrightarrow{n \to \infty} 0.
\end{equation}
On the other hand we have that
\begin{equation}
  \label{eq:massconv}
  \begin{aligned}
  \abs{\langle f_j,u \rangle_{\RL^2(M,\mu)} - \langle f_j^n,u
  \rangle_{\RL^2(M,\mu_n)}}
  &\le \abs{\langle f_j - J_n f_j^n,u \rangle_{\RL^2(M,\mu)}} + \\& \qquad + 
       \abs{ \langle J_n f_j^n,u \rangle_{\RL^2(M,\mu)} -  \langle f_j^n,u
       \rangle_{\RL^2(M,\mu_n)}}.
     \end{aligned}
\end{equation}
By Condition \cii, $f_j^n$ converges strongly in $\RL^2(M,\mu)$ so that the
first term on the righthand side converges to $0$ while Condition \civ implies
that the second term converges to $0$. Putting together \eqref{eq:energyconv},
\eqref{eq:smallext} and \eqref{eq:massconv} does yield that
\begin{equation}
\forall u \in \CW^{1,2}(M,\mu) \qquad \int_M \nabla f_j \nabla u \de v_g =
\lambda_j \int_M f_j u \de \mu.
\end{equation}

We can now prove that the limit sequence $f_j$ is orthonormal. Indeed,
\begin{equation}
  \langle f_j,f_k \rangle_{\RL^2(M,\mu)} = \langle J_n f_j^n,J_n f_k^n
  \rangle_{\RL^2(M,\mu)} + \langle f_j,f_k - J_n f_k^n\rangle_{\RL^2(M,\mu)} +
  \langle f_j - J_n f_j^n, f_k^n \rangle_{\RL^2(M,\mu)}.
\end{equation}
Strong convergence in $\RL^2(M,\mu)$, Conditions \cvi and \cv and the Cauchy-Schwarz
inequality imply that the the first term on righthandside converges to
$\delta_{jk}$ whereas the last two terms on the righthand side converge to $0$.

To prove that $\lambda_j(M,g,\mu) \le \lambda_j$, we use the space
$F_{j+1}=\mathrm{span}\{f_0,\ldots, f_j\}$ as a test space
in~\eqref{MeasureLaplace:def}. For any $f=\sum a_if_i\in F_j$ one has
$$
\frac{\displaystyle\int_M|\nabla f|^2_g\de v_g}{\displaystyle\int_M f^2\de \mu}
= \frac{\sum_{i=0}^j \lambda_i a_i^2}{\sum_{i=0}^j a_i^2}\leqslant
\lambda_j\frac{\sum_{i=0}^j a_i^2}{\sum_{i=0}^j a_i^2} = \lambda_j,
$$
where orthonormality of $\set{f_j}$ is used in the first equality. Finally, note that weak
convergence and convergence of the norms implies strong convergence, and it
follows from \eqref{eq:normef} that \eqref{eq:smallenergy} implies convergence of the norms.
\end{proof}

Our goal is now to provide conditions that can be verified  directly on the
measures $\mu_n$, $\mu$ to ensure convergence.

\begin{lemma}
  \label{lem:algebra}
  Let $u,v \in \RW^{1,2}(M)$. Then, if $d \ge
3$, then $uv \in \RW^{1,\frac{d}{d-1}}$ and
  \begin{equation}
    \norm{uv}_{\RW^{1,\frac{d}{d-1}}(M)} \le C_d \norm{u}_{\RW^{1,2}(M)} \norm{v}_{\RW^{1,2}(M)}
  \end{equation}
  For $d = 2$, $uv \in \RW^{1,p}$ for every $p < 2$, with the same norm estimate.
\end{lemma}

\begin{proof}
  It is sufficient to verify the claim for $\nabla(uv)$.
  Let $p = \frac{d}{d-1}$ ($d \ge 3$) or $p < 2$ ($d=2$). By the inequality $(a + b)^p \le 2^{p-1} (a^p + b^p)$ and Hölder's inequality
  with exponents $2/p$ and $2/(2-p)$, we have that
  \begin{equation}
    \begin{aligned}
      2^{1-p} \int_M \abs{\nabla (uv)}^p \de v_g &\le  \int_M \abs{\nabla u}^p
      \abs v^p + \abs{\nabla v}^p \abs u^p \de v_g \\
      &\le  
      \left(\int_M \abs{\nabla u}^2 \de v_g \right)^{\frac{p}{2}}\left(
        \int_M \abs v^{\frac{2p}{2-p}} \de v_g
      \right)^{1 - \frac p 2}+  \\ & \qquad + 
      \left(\int_M \abs{\nabla v}^2 \de v_g \right)^{\frac{p}{2}}\left(
        \int_M \abs u^{\frac{2p}{2-p}} \de v_g
      \right)^{1 - \frac p 2}.
  \end{aligned}
  \end{equation}
  By our conditions on $p$, the Sobolev embedding $\RW^{1,2}(M) \to
  \RL^{\frac{2p}{2-p}}(M)$ is bounded, so that
  \begin{equation}
      \left(\int_M \abs{\nabla u}^2 \de v_g \right)^{\frac{p}{2}}
      \left(
        \int_M \abs v^{\frac{2p}{2-p}} \de v_g
      \right)^{1 - \frac p 2}  \le 
      C \left(\int_M \abs{\nabla u}^2 \de v_g \right)^{\frac{p}{2}}
      \norm{v}_{\RW^{1,2}(M)}^p,
  \end{equation}
  and similarly swapping the roles of $u$ and $v$. This is precisely our claim.
\end{proof}
In dimension $2$, the target space for the product will be an
Orlicz--Sobolev space, recall \eqref{eq:orliczsob}.
\begin{lemma}
  \label{lem:algebrabis}
  Let $u,v \in \RW^{1,2}(M)$ and $d = 2$. Then, $uv \in \RW^{1,2,-1/2}(M)$ and there
  is $C$ such that
  \begin{equation}
    \norm{uv}_{\exp \RL^1(M)} \le C_2 \norm{u}_{\RW^{1,2}(M)}
    \norm{v}_{\RW^{1,2}(M)}.
  \end{equation}
\end{lemma}

\begin{proof}
  Just as in the previous case, it is sufficient to verify the claim for
  $\nabla(uv) = u \nabla v + v \nabla u$, and as such to verify that $u \nabla v
\in \RL^2 (\log L)^{-1/2}$. 
  Trudinger's theorem \cite{trudinger} states that for $d = 2$, $\RW^{1,2}(M)$ embeds
  continuously in $\exp \RL^2(M)$, so that $u \in \exp \RL^2(M)$, and by
  assumption $\nabla v \in \RL^2(M)$. Taking 
  \begin{equation}
    R(t) = \frac{t^2}{\log(2 + t)}, \qquad M(t) = t^2, \qquad N(t) = \exp(t^2) -
    1
  \end{equation}
  in \cite[Theorem 1]{ando}, we have the equivalence between the following
  statements:
  \begin{itemize}
    \item[$\bullet$] for all $u \in \exp \RL^2(M)$ and $v \in \RL^2(M)$, $uv \in
      \RL^2(\log L)^{-1/2}$;
    \item[$\bullet$] there is $a,b>0$ such that
  \begin{equation}
    \label{eq:neededforproduct}
    \frac{(a s t)^2}{\log(2 + (a s t))} \le s^2 + \exp(t^2) - 1\qquad \forall
    s,t \ge b.
  \end{equation}
  \end{itemize}
  We verify that this second statement holds for $b = 1$. On the one hand,
  \begin{equation}
    s^2 \ge \frac{\exp(2 a^2 t^2)}{a^2t^2} \quad \Longrightarrow \quad
    \frac{(ast)^2}{\log (2
    + (a s t))} \le s^2,
  \end{equation}
  while on the other hand 
  \begin{equation}
    s^2 \le \frac{\exp(t^2)}{a^2 t^2} \text{ and } t \ge 1 \quad \Longrightarrow \quad 
    \frac{(ast)^2}{\log (2
    + (a s t))} \le \exp(t^2) - 1.
  \end{equation}
  Choosing $a \le 2^{-1/2}$ ensures that the inequality
  \eqref{eq:neededforproduct} holds for all $s, t \ge 1$. Therefore,
  \begin{equation}
    \norm{u \nabla v}_{\RL^2 (\log L)^{-1/2}} \le C \norm{u}_{\exp L^2}
    \norm{\nabla v}_{\RL^2} \le C' \norm{u}_{\RW^{1,2}} \norm v_{\RW^{1,2}}.
  \end{equation}
  the same holds swapping $u$ and $v$ and our claim follows.
\end{proof}

Let $X$ be a completion of $\RC^\infty(M)$ under some norm $\norm{\cdot}_X$. We
can interpret measures as bounded linear functionals on $X$ as $\langle \mu,f\rangle_X = \int_M f \de
\mu$ as long as
\begin{equation}
  \norm{\mu}_{X^*} = \sup_{f \in \RC^\infty(M) \setminus \set{0}} \frac{\abs{\int_M f \de
  \mu}}{\norm{f}_X}
\end{equation}
is finite.

\begin{prop}
  \label{prop:weakenough}
  Suppose that $\Omega_n \subset M$ is a sequence of domains in $M$ and $\mu_n,
  \mu$ are Radon measures on $M$ satisfying conditions \csi--\csvi. If $d \ge
  3$, suppose that $\mu_n \to \mu$ in $\RW^{1,\frac{d}{d-1}}(M)^*$. If $d =
  2$, suppose that $\mu_n \to \mu$ in $\RW^{1,2,-1/2}(M)^*$. Then, Conditions \csiii--\csv are satisfied by the
  eigenfunctions. In particular,
  \begin{equation}
    \lim_{n \to \infty} \lambda_j(\Omega_n,g,\mu_n) = \lambda_j(M,g,\mu).
  \end{equation}
\end{prop}

\begin{remark}
  Since $M$ is compact, convergence in $\RW^{1,p}(M)^*$ implies convergence in
  $\RW^{1,q}(M)^*$ for every $q > p$, so that in practice one can verify this
  criterion for any $p < \frac{d}{d-1}$. Often the case $p = 1$
  provides easier computations. We also remark that if $\mu_n \to \mu$ in
  $\RW^{1,p}(M)^*$, 
\end{remark}

\begin{proof}
  Let us first assume that $d \ge 3$ and put $p = \frac{d}{d-1}$, or $d = 2$,
  $p<2$. We first observe that the trace operators $T_{2}^{\mu_n}$ are bounded,
  uniformly in $n$. Indeed, by Lemma \ref{lem:algebra} for every $u \in
  \RW^{1,2}(M)$,
  \begin{equation}
    \begin{aligned}
      \int_M u^2 (\rd \mu_n - \rd \mu) &\le \norm{u^2}_{\RW^{1,p}(M)} \norm{\mu_n
      - \mu}_{\RW^{1,p}(M)^*} \\
      &\le C_p \norm{u}_{\RW^{1,2}(M)}^2 \norm{\mu_n -
      \mu}_{\RW^{1,p}(M)^*},
    \end{aligned}
  \end{equation}
  so that
  \begin{equation}
    \begin{aligned}
      \int_M u^2 \de \mu_n &= \int_M u^2 (\rd \mu_n - \rd \mu) + \int_M u^2 \de
      \mu \\
      &\le \left(C_p \norm{\mu_n - \mu}_{\RW^{1,p}(M)^*} +
      \norm{T_2^\mu}^2\right) \norm{u}_{\RW^{1,2}(M)}^2.
    \end{aligned}
  \end{equation}
  To verify \csiv we note that by Lemma
  \ref{lem:algebra} one has
  \begin{equation}
    \begin{aligned}
      \int_{M} (J_n f_j^n)^2 \de \mu &\le \norm{(J_n f_j^n)^2}_{\RW^{1,p}(M)}
      \norm{\mu}_{\RW^{1,p}(M)^*} \\
      &\le C_p \norm{\mu}_{\RW^{1,p}(M)^*}
      \norm{J_n f_j^n}_{\RW^{1,2}(M)}^2  \\
      &\le C_p \norm{\mu}_{\RW^{1,p}(M)^*}\left(1 +  \norm{T^{\mu_n}_2}\right)
      \norm{J_n f_j^n}_{\CW^{1,2}(M,\mu_n)}^2
  \end{aligned}
  \end{equation}

  We have that
  \begin{equation}
    \label{eq:ef2*part2}
    \abs{\int_M J_n f_j^n u (\rd \mu_n - \rd \mu)} \le \norm{J_nf_j^n u}_{\RW^{1,p}(M)}
    \norm{\mu_n - \mu}_{\RW^{1,p}(M)^*}.
  \end{equation}
  This goes to $0$ by Lemma \ref{lem:algebra} and convergence $\mu_n \to \mu$ in
  $\CW^{1,p}(M)^*$, so that Condition \csiv is satisfied.

  Finally, using Lemma \ref{lem:algebra} one last time, we have that
  \begin{equation}
    \abs{\int_M J_n f_j^n J_n f_k^n (\rd \mu_n - \rd \mu)} \le \norm{f_k^n
    f_j^n}_{\RW^{1,p}(M)} \norm{\mu_n - \mu}_{\RW^{1,p}(M)^*} \xrightarrow{n \to
    \infty} 0,
  \end{equation}
  so that Condition \csv is indeed satisfied.

  The case $d = 2$ and $\mu_n \to \mu$ in $\RW^{1,2,-1/2}(M)^*$ follows from replacing Lemma \ref{lem:algebra} by Lemma
  \ref{lem:algebrabis}.
\end{proof}

Verifying that $\mu_n \to \mu \in \RW^{1,p}(M)^*$ is in principle a global
question (or at the very least the local character of it should be verified
independent of $n$). Following the ideas set out in \cite{GHL,GL} we provide the
following blueprint for
verifying convergence in an effective way that is based on Lemma
\ref{lem:integralidentitybis}. This lemma implies, amongst other things, that if
$\mu$ is $p$-admissible then there exists $\phi_n \in \RW^{1,p'}(M)$ such that
\begin{equation}
  \langle \mu_n - \mu,f\rangle_{\RW^{1,p}(M)} = \left( \frac{\mu_n(M)}{\mu(M)} -
  1\right)\langle \mu,f\rangle_{\RW^{1,p}(M)} + \int_M \nabla \phi_n \cdot
  \nabla f \de v_g.
\end{equation}
The first term is easily seen to converge to $0$ uniformly for $\norm
f_{\RW^{1,p}(M)} \le 1$. However, the estimates we obtained on $\phi_n$ in Lemma
\ref{lem:integralidentitybis} are not on their face strong enough to guarantee
convergence. In Section \ref{sec:hom} we get over this hurdle by partitioning $M$
into an almost disjoint union $M = \bigcup_{z \in
I_n} Q_z^n$. This allows us to write
\begin{equation}
  \langle \mu_n - \mu,f \rangle_{\RW^{1,p}(M)^*} = \sum_{z \in I} \langle \mu_n
  - \mu,f \rangle_{\RW^{1,p}(Q_z^n)^*}.
\end{equation}
Using Lemma \ref{lem:integralidentitybis} in every $Q_z^n$ provides us with an effective
mean of proving that this converges to $0$. In view of estimate
\eqref{eq:phiestimate}, if $\mu_n(Q_z^n)$ is comparable for every $z$ then by
H\"older's inequality
\begin{equation}
  \sum_{z \in I} \mu_n(Q_z^n)^{1/p'} \norm{f}_{\RW^{1,p}(Q_z^\eps)} \ll
\norm{f}_{\RW^{1,p}(M)}
\end{equation}
so that in principle one will need to prove only that the $p$-Poincar\'e
constants of $Q_z^\eps$ are uniformly bounded and that the traces $T_p^{\mu_n}$
restricted to $Q_z^\eps$ converge to $0$. Exploiting the potential lack of scale
invariance in the defining equation for $\phi_n$ is often key for this. For a
concrete application of this scheme, see the proof of Proposition \ref{prop:homstek}.

\section{First examples of spectrum convergence}

\label{sec:1applications}

In this section we collect several applications of the setup presented in the
previous section. Most of the results in this section are generalisations of
known results either to a manifold context or to higher dimensions.
\subsection{Convergence for \texorpdfstring{$\RL^p$}{Lp} densities}  The case
$d=2$, $\beta_n\in \RL^p$, $p > 1$ has previously appeared
in~\cite[Lemma 6.2]{KNPP2}.
\begin{prop}
\label{L_infty_approx:prop}
Let $\beta_n$ be a sequence of non-negative densities converging in
$\RL^{\frac{d}{2}} (\log
\RL)^a(M)$ to a non-negative density $\beta$, where $a = 0$ for $d \ge 3$ and $a =
1$ for $d = 2$.
Then $\lambda_k(M,g,\beta_n\de v_g)\to \lambda_k(M,g,\beta\de v_g)$ as $n\to \infty$.
\end{prop}
\begin{proof}
Conditions \ci--\cvi are respected, admissibility following from Example
\ref{ex:lpdensity}. 
Let $u \in \RW^{1,\frac{d}{d-1}}(M)$ for $d \ge 3$. Then by the Sobolev
embedding $\RW^{1,\frac{d}{d-1}}(M) \to \RL^{\frac{d}{d-2}}(M)$ and H\"older's
inequality with exponents $\frac{d}{d-2}$ and $\frac{d}{2}$,
  \begin{equation}
    \abs{\int u (\beta_n - \beta) \de v_g} \ll_{d,M}
    \norm{u}_{\RW^{1,\frac{d}{d-1}}(M)} \norm{\beta_n - \beta}_{\RL^{d/2}}.
  \end{equation}
  We deduce that $\beta_n \rd v_g \to \beta \rd v_g$ in
  $\RW^{1,\frac{d}{d-1}}(M)^*$, so that by Proposition \ref{prop:weakenough} the
  eigenvalues converge. For $d = 2$, proceed the same way but with the pairing
  of the spaces
  $\exp \RL^1(M)$ and $\LLL(M)$, along with the optimal Sobolev embedding
  $\RW^{1,2,-1/2}(M) \to \exp \RL^1(M)$, see \cite[Example 1]{cianchi}.
\end{proof}

\subsection{Approximation of eigenvalues of measures supported on a hypersurface.}

Let $(M,g)$ be a compact Riemannian manifold. Let $\Sigma\subset M$ be a compact, not necessarily connected, codimension $1$ smooth submanifold without boundary and
$\rho\in C(\Sigma)$ be a non-negative density on $\Sigma$. Assume $\Sigma=\Sigma_i\sqcup\Sigma_b$, where $\Sigma_i\cap\del M=\varnothing$ and $\Sigma_b$ is either empty or coincides with $\del M$.  Let $N_{\eps,b}$ be an
$\eps$-tubular neighbourhood of $\Sigma_i$. For sufficiently small $\eps$
the
exponential map $\exp_{\Sigma_i}$ can be used to identify $N_{\eps,i}$ with
$N^\eps\Sigma_i$, the $\eps$-ball in the normal bundle of $\Sigma_i$. 
Similarly, if $\Sigma_b=\del M$ is not empty, its $\eps$-tubular neighbourhood $N_{\eps,b}$ can be identified with $\Sigma_b\times[0,\eps]$ using $\exp_{\Sigma_b}$. If $n$ is an outward unit normal then 
we define 
\begin{equation}
\rho_\eps(y) = 
  \begin{cases}
    \frac{1}{2\eps}\rho(x) & \text{if } y=\exp_x(w)\in N_{\eps,i}, \,\,(x,w)\in N^\eps\Sigma_i \\
    \frac{1}{\eps}\rho(x) &\text{if } y =\exp_x(-tn)\in N_{\eps,b},\,\,(x,t)\in \Sigma_b\times[0,\eps] \\
    0  & \text{otherwise} .
  \end{cases}
\end{equation}

The next theorem says that we can approximate the eigenvalues of weighted Steklov
or transmission problems as in Example \ref{ex:transmission} using weighted Laplace
eigenvalues. When $\Sigma = \Sigma_b = \del M$, our construction is similar to
the one found for domains in $\R^d$ in \cite{LambPro}.

\begin{theorem}
\label{Steklov_approx:thm}

Let $\,\mathrm d A^\Sigma_g$ be the volume measure on $\Sigma$. Then one has 
$$
\lambda_k(M,g,\rho_\eps\rd v_g)\to \lambda_k\left(M,g, \rho \rd A^\Sigma_g\right)
$$
as $\eps\to 0$. 
\end{theorem}
\begin{proof}
 We give the proof in the case $\Sigma=\Sigma_i$, the other case is analogous.
The conditions \ci--\cvi are obviously satisfied. We claim that for any $u\in \RW^{1,1}(M)$ one has
\begin{equation}
\label{W11:estimate}
\left| \int_M u\rho_\eps\de v_g - \int_\Sigma u\rho\de A^\Sigma_g\right| \le C
\norm{u}_{\RW^{1,1}(N_\eps)}
\end{equation}
for $\eps$ small enough. In particular, for any $p > 1$ and $u \in \RW^{1,p}(M)$ one
has by H\"older's inequality
\begin{equation}
\label{W12:estimate}
\abs{\langle \rho_\eps \rd v_g - \rho \rd A_g^\Sigma, u \rangle} \ll
  \Vol(N_\eps)^{1/p'} \norm{u}_{\RW^{1,p}(M)},
\end{equation}
which goes to $0$ since $\Vol(N_\eps) \to 0$. Thus, proving \eqref{W11:estimate}
implies that $\rho_\eps \rd v_g \to \rho \rd A_g^\Sigma$ in $\RW^{1,p}(M)^*$,
and Proposition \ref{prop:weakenough} implies that the eigenvalues converge.

We use coordinates $(x,w)$ on $N_\eps$ induced by the identification with
$N^\eps\Sigma$ via the exponential map. Let  $\xi(x,w)\de w\de A^\Sigma_g$ be
the volume measure on $N_\eps$, where we denoted the pullback $\pi^*\de
A^\Sigma_g$ simply by $\de A^\Sigma_g$, $\pi\colon N^\eps\Sigma\to\Sigma$. Let
also $\zeta(x) = \int_{N^\eps_x\Sigma}\xi(x,w)\de w$ be the fiber integral of
$\xi(x,w)$. Since the differential of the exponential map at the origin is equal to identity, one has that
\begin{equation}
\label{zxi:ineq}
|1-\xi(x,w)| = O(\eps)\qquad |1-\zeta(x)| = O(\eps)
\end{equation}
as $\eps\to 0$. Then one has
\begin{equation}
\label{approx:ineq}
\begin{aligned}
\left| \int_M u\rho_\eps\de v_g - \int_\Sigma u\rho\de A^\Sigma_g\right|
&\le
\frac{\norm{\rho}_{L^\infty}}{2\eps}\int_\Sigma\int_{N^\eps_x\Sigma}|u(x,w)\xi(x,w)-u(x,0)\zeta(x)|\de
w\de A^\Sigma_g \\ &\le 
\frac{C}{\eps}\int_\Sigma\int_{N^\eps_x\Sigma}|u(x,w)-u(x,0)|\zeta(x)\de
w\de A^\Sigma_g \\&\quad+
\frac{C}{\eps}\int_\Sigma\int_{N^\eps_x\Sigma}|u(x,w)||\zeta(x)-\xi(x,w)|\de
w\de A^\Sigma_g \\
&\le\frac{C'}{\eps}\int_\Sigma\int_{N^\eps_x\Sigma}|u(x,w)-u(x,0)|\de w\de
A^\Sigma_g \\&\quad + C''\int_\Sigma\int_{N^\eps_x\Sigma}|u(x,w)|\de w\de A^\Sigma_g,
\end{aligned}
\end{equation}
where we used~\eqref{zxi:ineq} in the last step. For a fixed $x$ the inside integral in the first term is a $1$-dimensional integral that can be estimated as follows, 
$$
\begin{aligned}
\int_{-\eps}^\eps|u(x,t)-u(x,0)|\de t &= \int_{-\eps}^\eps
\left|\int_0^tu_t(x,s)\de s\right|\de t \\&\le
\int_{-\eps}^\eps\int_{-\eps}^\eps |u_t(x,s)| \de s\de t \\&= 2\eps
\int_{-\eps}^\eps |u_t(x,s)| \de s
\end{aligned}
$$

Putting this back into~\eqref{approx:ineq} and using~\eqref{zxi:ineq} completes
the proof of estimate \eqref{W11:estimate}.

\end{proof}

\subsection{Application to shape optimisation in dimension 2}

We can now prove Theorem \ref{thm:sigmaklambdak}, with the
following proposition.

\begin{prop}
Let $(M,g)$ be a compact Riemannian surface and let $\,\Omega\subset M$ be a smooth domain such that $\del\Omega\cap\del M$ is either empty or equal to $\del M$. Then 
$$
\normalstek{k}(\Omega,g) \le \optimlap{k}(M,[g]).
$$
\end{prop}
\begin{proof}
We first note that in dimension $2$ for any smooth $\rho>0$ one has that
$\lambda_k(M,g,\rho\de v_g)$ coincides with the classical Laplacian eigenvalues
$\lambda_k(M,\rho g, \de v_{\rho g})$ of the conformal metric $\rho g$. Since
smooth functions are dense in $\RL^p$ for every $p \in [1,\infty)$, Theorem~\ref{Steklov_approx:thm} and Proposition~\ref{L_infty_approx:prop} imply that 
$$
\normallap{k}(M,g,\de A^\Sigma_g)\le \optimlap{k}(M,[g]),
$$
where $\de A^\Sigma_g$ is the surface measure of $\Sigma=\partial \Omega$. At the same time, a comparison of the Rayleigh quotients for $\lambda_k(M,g,\de A^\Sigma_g)$ and $\sigma_k(\Omega,g)$ yields the inequality
$$
\normalstek{k}(\Omega,g)\le \normallap{k}(M,g,\de A^\Sigma_g).
$$
 
\end{proof}

\begin{remark}
\label{rem:BucurNahon}
  It is clear from our constructions that as soon as a measure $\mu$ on
  $M$
  is limit in $\RW^{1,\frac{d}{d-1}}(M)^*$ of measures of the form $\beta \de x$ respecting
  conditions \ci--\cvi, we obtain similarly to the last
  proposition
  \begin{equation}
    \lambda_k(\Omega,g,\mu) \mu(M) \le \Lambda_k(M,[g]).
  \end{equation}
\end{remark}

  \section{Homogenisation}
\label{sec:hom}
In this section, we fix a bounded domain $\Omega \subset \R^d$ with Lipschitz boundary
and we put $M = \overline \Omega$. Let $\beta \in \RC(M)$ be
nonnegative, and non trivial,  $g_0$ be the flat metric and $\rd A$ be the boundary measure on $\del M$.
\subsection{Construction of perforated sets}
\label{section:homoconst}
 We
construct domains $\Omega^\eps \subset M$ in the spirit of
deterministic homogenisation theory.
For $z \in \Z^d$, consider the cubes 
$$Z_z^\eps := \eps z +
\left[-\frac \eps 2,\frac \eps 2 \right]^d \subset \R^d$$ and define
\begin{equation}
  \BI^\eps := \set{z \in \Z^d : Z_z^\eps \subset M}. 
\end{equation}
For $\alpha > d-1$, we set
\begin{equation}
  \label{eq:reps}
  r_{z}^\eps :=
  \left(\frac{\eps^{\alpha}}{a_d} \beta(\eps z)\right)^{\frac{1}{d-1}},\qquad B_z^\eps
  = B(\eps z,r_z^\eps), \qquad \text{and} \qquad Q_z^\eps := Z_z^\eps
  \setminus B_z^\eps,
\end{equation}
where $a_d$ is the area of the unit sphere in $\R^d$ and where, by convention,
we put $B(x,0) = \varnothing$ for any $x \in \R^d$. We set as well
\begin{equation}
  \BR^\eps := M \setminus \bigcup_{z
 \in \BI^\eps} Z_z^\eps, \qquad \tilde \BI^\eps := \set{z \in \BI^\eps :
 \beta(\eps z) \ne  0} \qquad \text{and} \qquad \BB^\eps := \bigcup_{z \in \tilde
 \BI^\eps} B_z^\eps;
\end{equation}
and finally
\begin{equation}
  \Omega^\eps := M \setminus \BB^\eps \qquad \text{and} \qquad
  \mu^\eps_\alpha :=
  \iota_* \rd A^\eps,
\end{equation}
where $\iota_* \rd A^\eps$ is the pushforward by inclusion of the natural
boundary measure on $\Omega^\eps$. For any measure $\xi$, we define the normalised
measure $\bar \xi :=
\xi(M)^{-1} \xi$. 
\begin{itemize}
  \item[$\bullet$] For all $\eps, z$, 
    \begin{equation}
      (r^\eps_z)^{d-1} \ll  \eps^\alpha\max_{x \in M} \beta(x).
    \end{equation}
  \item[$\bullet$] The number of holes satisfies $\#\tilde \BI^\eps \ll_M \eps^{-d}$
    as $\eps \to 0$.
  \item[$\bullet$] The total boundary area of the holes satisfies the asymptotic
    relationship
      \begin{equation}
        \label{eq:totalarea}
      \CH^{d-1}(\del \BB^\eps) = \sum_{z \in \tilde \BI^\eps} a_d
        (r_z^\eps)^{d-1} \sim \eps^{\alpha - d}
        \int_{M} \beta \de x.
      \end{equation}
      while the total volume of the holes satisfies
    \begin{equation}
      \label{eq:volumeholes}
    \Vol(\BB^\eps) = \sum_{z \in \tilde \BI^\eps} d a_d (r_z^\eps)^{d} =
    \bigo[M,\beta]{\eps^{\frac{d\alpha}{d-1}-d}},
    \end{equation}
      In particular, Condition \csi is satisfied with
      \begin{equation}
        \label{eq:barmu}
        \bar \mu^\eps_\alpha \wks \bar \mu_\alpha := \begin{cases}
          \bar{\beta \rd v_g} & \text{if } d - 1 < \alpha < d, \\
          \bar{\beta \rd v_g + \iota_* \rd A} & \text{if } \alpha = d,
          \\
          \bar{\iota_* \rd A} & \text{if } \alpha > d ;
      \end{cases}
    \end{equation}
      and
      \begin{equation}
        \rd x\big|_{\Omega^\eps} \wks \rd x\big|_M.
      \end{equation}
    \item[$\bullet$] It is a standard fact that on Lipschitz domains the trace maps
    $T_2^{\mu^\eps}$ and the
    Sobolev embeddings $T_2^{\mu}$ are compact, and that the first Steklov and Neumann
      eigenvalues are always positive so that Condition
      \csii is met in both cases.
    \item[$\bullet$] The set $\BR^\eps$ is a subset of a $\sqrt d \eps$-collar
      neighbourhood of $\del M$, as such $\Vol(\BR^\eps) = \bigo[d,M]{\eps}$.
    \item[$\bullet$] Denoting by $J^\eps : \CW^{1,2}(\Omega^\eps,\mu^\eps_\alpha) \to
      \CW^{1,2}(\Omega,\mu^\eps_\alpha)$ the map
      extending harmonically in $\BB^\eps$, we have that $J^\eps$ is bounded
      independently of $\eps$, see \cite[Example 1, page 40]{RauchTaylor}.
      Condition \csvi is therefore satisfied.
\end{itemize}

\begin{remark}
\label{manifold:remark}
  To obtain Theorem~\ref{manifoldhomo:thm} we note that it is possible to achieve a similar setting by removing geodesic
  balls of radius $r_\eps$ around a maximal $\eps$-separated subset of a
  Riemannian manifold $M$.
  See \cite{AnnePostJST} and
  \cite{GL} for similar constructions in the context of the Neumann and Steklov
  problems, respectively. This makes it possible to directly extend the
  statements to the situation where $\Omega$ is a bounded domain with Lipschitz
  boundary in a complete manifold $\tilde M$, the implicit
  constants then depending on the metric of $\tilde M$ restricted to $\Omega$. We keep the periodic description here to emphasise the
  fact that we do not need the Riemannian setting in order to obtain large
  normalised Steklov eigenvalues.
\end{remark}

With the notation introduced above, we may now state the main theorem of this
section.

\begin{theorem}
  \label{thm:homo}
  For all $j \in \N$, and 
  \begin{equation}
    \label{eq:alphacond}
  \alpha > d - 1
\end{equation}
   the Steklov
  eigenvalues of the perforated domains $\Omega^\eps$ satisfy
  \begin{equation}
    \label{eq:Steklov}
    \sigma_j(\Omega^\eps)
    \frac{\CH^{d-1}(\del\Omega^\eps)}{\Vol(\Omega^\eps)^{\frac{d-2}{d}}} =
    \normallap{j}(\Omega^\eps,g_0,\bar \mu^\eps_\alpha)   \xrightarrow{\eps \to 0}
    \normallap{j}(M,g_0,\bar \mu_\alpha),
  \end{equation}
  where $\bar \mu$ is defined in \eqref{eq:barmu} and the associated
  eigenfunctions extended to $M$ converge strongly in $\RW^{1,2}(M)$. The Neumann eigenvalues
  satisfy
  \begin{equation}
    \label{eq:Neumann}
    \lambda_j(\Omega^\eps,g_0)\Vol(\Omega^\eps)^{2/d} = \normallap{j}(\Omega^\eps,g_0,\rd
  x) \xrightarrow{\eps \to 0}
  \normallap{j}(M,g_0,\rd x),
  \end{equation}
  and the associated eigenfunctions extended to $M$ converge weakly in
  $\RW^{1,2}(M)$.
\end{theorem}
The proof of Theorem \ref{thm:homo} is split into
two parts: Proposition \ref{prop:homneu} where the convergence of the Neumann
eigenvalues is shown and Proposition \ref{prop:homstek} where we prove
convergence of the Steklov eigenvalues. In both cases, we prove that the
associated measures converge in $\RW^{1,p}(M)^*$,
for all $p > 1$. 
In the construction of perforated sets, we already have
that conditions \csi--\csvi are satisfied, so that by Proposition \ref{prop:weakenough},
this is enough to obtain eigenvalue and eigenfunction convergence. For the
Steklov problem, we observe that Condition \eqref{eq:smallenergy}
follows directly from \cite[Lemma 12]{GHL} so that strong convergence of the
eigenfunctions also follows if we prove the appropriate convergence of the
measures.

  We note that \eqref{eq:Neumann} could be
deduced by an appropriate modification of the proofs in \cite{RauchTaylor} or
\cite{AnnePostJST}, however this would require introducing new concepts whereas
the results from Sections \ref{section:variationaleigenvalues} and
\ref{variationalev:sec} can
prove both convergence of the Neumann and Steklov eigenpairs at the same time. This also puts an emphasis on the fact that it is
achieved for the same domains. 

\begin{prop}
  \label{prop:homneu}
  As $\eps \to 0$, the measures $\rd x\big|_{\Omega^\eps}$ converge to $\rd
  x\big|_{M}$ in $\RW^{1,p}(M)^*$ for every $p \in (1,\infty)$. In particular,
  the Neumann eigenpairs for $\Omega^\eps$ converges to those of $\Omega$.
\end{prop}

\begin{proof}
  For $f \in \RW^{1,p}(M)$, we have that
  \begin{equation}
    \begin{aligned}
      \abs{\langle \rd x\big|_{\Omega^\eps} - \rd x\big|_M, f \rangle} &= \abs{\int_{\BB^\eps} f
      \de x} \\
      &\le \Vol_{g_0}(\BB^\eps)^{\frac{p-1}{p}} \norm{f}_{\RL^{p}(\BB^\eps)} \\
      &\le \Vol_{g_0}(\BB^\eps)^{\frac{p-1}{p}} \norm{f}_{\RW^{1,p}(M)}.
    \end{aligned}
  \end{equation}
  By \eqref{eq:volumeholes}, this last line goes to $0$ as $\eps \to 0$. 
\end{proof}

\subsection{Convergence of the Steklov eigenpairs}

Before proving convergence of the Steklov eigenpairs, we require the following
useful lemma. 
  \begin{lemma}
    \label{lem:tracew11l1}
    Let $0 < r \le R \le 1$, and $p \in (1,d)$. Then, there exists $C_{p,d} > 0$
    such that for all $f \in
    \RW^{1,p}(B(0,R))$, 
\begin{equation}
  \norm{f}_{\RL^p(r \S^{d-1})}^p \le C_{p,d} \max\set{r^{d-1} R^{-d},r^{p-1}}
  \norm{f}_{\RW^{1,p}(B(0,R))}^p.
\end{equation}
  \end{lemma}
  \begin{proof}
    By density of smooth functions in $\RW^{1,p}(B(0,R))$ it is sufficient to
    prove the inequality for smooth $f$. Let $\tilde f \in \RW^{1,p}(B(0,R))$ be the radially constant function given
    by $\tilde f(\rho,\theta) = f(r,\theta)$, since $p < d$ we can assign any
    value of $\tilde f$ at $0$. Set $F := f - \tilde f \in
    \RW^{1,p}(B(0,R))$ and observe that $F$ vanishes on $\del B(0,r)$, and
    that $\del_\rho F = \del_\rho f$. We
    directly compute that
    \begin{equation}
      \label{eq:friedrichs}
      \begin{aligned}
      \norm{f}_{\RL^p(r \S^{d-1})}^p &= d r^{d-1} R^{-d} \norm{\tilde
      f}_{\RL^p(B(0,R))}^p \\&\le 2^{p-1} d r^{d-1} R^{-d} \left(
      \norm{f}_{\RL^p(B(0,R))}^p + \norm{F}_{\RL^p(B(0,R))}^p \right).
    \end{aligned}
    \end{equation}
    To conclude, we will bound the norm of $F$ with a radial Friedrichs'
    inequality. By simple integration and H\"older's inequality we have that for
    every $\rho \in (0,R)$ and $\theta \in \S^{d-1}$
    \begin{equation*}
    \begin{aligned}
      \abs{F(\rho,\theta)}^p = \abs{\int_r^\rho \del_s f(s,\theta)
      \de s}^p &\le  \abs{\int_r^\rho s^{\frac{1-d}{p-1}}
      \de s}^{p-1} \int_r^\rho\abs{\del_s f(s,\theta)}^p s^{d-1}\de s.
   \end{aligned}
    \end{equation*}
    Integrating both sides of this inequality on $B(0,R)$ tells us that since $p
    < d$
    \begin{equation}
      \label{eq:friedbis}
      \begin{aligned}
      \norm{F}_{\RL^p(B(0,R))}^p &\le \int_0^R  \rho^{d-1}\abs{\int_r^\rho
      s^{\frac{1-d}{p-1}}\de s}^{p-1}  \norm{\del_\rho
      f}_{\RL^p(B(0,\rho))}^p \de \rho \\
      &\le \frac{p-1}{d-p} \norm{\del_\rho
      f}_{\RL^p(B(0,R))}^p \int_0^R \rho^{p-1} \abs{1 -
      \left(\frac{r}{\rho}\right)^\frac{p-d}{p-1}}^{p-1} \de \rho
    \end{aligned}
    \end{equation}
    This integral can be split into regions where $\rho \le a := 2^{\frac{d-p}{p-1}}r$
    and $a \le \rho \le R$. In the first region, the integral is bounded by
    $2^{p-1}$. In the second region, we have that
    \begin{equation}
      \begin{aligned}
        \int_a^R \rho^{p-1} \abs{1 - \left( \frac r \rho
        \right)^{\frac{p-d}{p-1}}}^{p-1} \de \rho &\le \frac 1 2 \int_a^R
        \rho^{d-1} r^{p-d} \de \rho \\
        &\le \frac{1}{2d} R^{-d} r^{p-d}. 
    \end{aligned}
    \end{equation}
    Inserting this estimate into \eqref{eq:friedbis} and then \eqref{eq:friedrichs} yields our claim.
\end{proof}

The main purpose of this section is to prove the following proposition
\begin{prop}
  \label{prop:homstek}
  As $\eps \to 0$, the measures $\bar \mu^\eps_\alpha \to \bar \mu_\alpha$ in
  $\RW^{1,p}(M)^*$ for all $p > 1$. 
\end{prop}
\begin{proof}
  Without loss of generality, by monotonicity of the dual spaces
  $\RW^{1,p}(M)^*$ we assume that $ p < 2$.

  For $f \in \RW^{1,p}(M)$  we have the decomposition
  \begin{equation}
    \label{eq:decomp}
    \langle \bar \mu^\eps_\alpha - \bar \mu_\alpha,f\rangle_{\RW^{1,p}(M)} =
    \frac{\bone_{\set{\alpha \le d}}}{\mu_\alpha(M)} \int_{\BR^\eps}
    \beta f
    \de x
     + \left(\frac{1}{\mu_\alpha^\eps(M)} -
    \frac{\bone_{\set{\alpha \ge d}}}{\mu_\alpha(M)} \right)\int_{\del M} f \de A+
    \sum_{z \in
    \BI^\eps} \langle \bar \mu^\eps_\alpha - \bar \mu_\alpha,f
    \rangle_{\RW^{1,p}(Z_z^\eps)}
  \end{equation}
  We first observe that
  \begin{equation}
    \int_{\BR^\eps} \beta f \de x \le \norm{\beta}_{\RC^0(M)}
    \norm{f}_{\RL^p(M)} \Vol(\BR^\eps)^{\frac{p-1}{p}} \ll_{M,\beta} \eps^{\frac{p-1}{p}}.
  \end{equation}
  We also have that
  \begin{equation}
    \lim_{\eps \to 0}  \left(\frac{1}{\mu^\eps(M)} -
    \frac{\bone_{\set{\alpha \ge d}}}{\mu(M)} \right)\int_{\del M} f \de A = 0;
  \end{equation}
  for $\alpha \ge d$ this follows from the fact that $\mu^\eps_\alpha(M)
  \xrightarrow{\eps \to 0} \mu_\alpha(M)$ whereas for $d-1 < \alpha < d$ this follows
  from $\mu^\eps_\alpha(M) \xrightarrow{\eps \to 0} \infty$. We are now left only with
  the sum term in \eqref{eq:decomp}. 
  
  In the case where $\alpha > d$, we have that $\bar \mu_\alpha$ is supported on
  $\del M$ so that by H\"older's inequality on sums that
  \begin{equation}
    \begin{aligned}
    \sum_{z \in \BI^\eps} \abs{\langle \bar \mu_\alpha^\eps - \bar \mu_\alpha,f
    \rangle_{\RW^{1,p}(Z_z^\eps)}} &\le \frac{1}{\mu_\alpha^\eps(M)} \sum_{z \in
    \BI^\eps} \int_{\del B_z^\eps} \abs f  \de A \\
  &\ll \sum_{z \in \BI^\eps} \CH^{d-1}(\del B_z^\eps)^{\frac{p-1}{p}}
  \norm{f}_{\RL^p(\del B_z^\eps)}  \\
  &\ll_{\beta,M} \eps^{(\alpha - d) \frac{p-1}{p}} \left(\sum_{z \in \BI^\eps}
  \norm{f}_{\RL^p(\del B_z^\eps)}^p \right)^{1/p};
  \end{aligned}
  \end{equation}
  this goes to $0$ as $\eps \to 0$. Indeed, since $\alpha > d$ we have that $(r_z^\eps)^{d-1}
  \eps^{-d}$ remains uniformly bounded as $\eps \to 0$, so that by Lemma \ref{lem:tracew11l1}
  \begin{equation}
    \sum_{z \in \BI^\eps} \norm{f}_{\RL^p(\del B_z^\eps)}^p \le \sum_{z \in
    \BI^\eps} \norm{f}^p_{\RW^{1,p}(Z_z^\eps)} \le \norm{f}^p_{\RW^{1,p}(M)}.
  \end{equation}

  When $d-1 < \alpha \le d$, we split the sum into $z \in \BI^\eps \setminus \tilde \BI^\eps$ and
  $z \in \tilde \BI^\eps$. In the first case, we have that
  \begin{equation}
    \begin{aligned}
\abs{    \sum_{z \in \BI^\eps \setminus \tilde \BI^\eps} \langle \bar
\mu^\eps_\alpha - \bar \mu_\alpha,f\rangle_{\RW^{1,p}(Z_z^\eps)} }&=
\frac{1}{\mu_\alpha(M)}\abs{\sum_{z \in
\BI^\eps\setminus \tilde \BI^\eps} \int_{Z_z^\eps} \beta f \de x} \\
&\le \sup_{z \in \BI^\eps} \sup_{x \in Z_z^\eps} \abs{\beta(x)}
\norm{f}_{\RL^1(M)}.
\end{aligned}
  \end{equation}
  By uniform continuity of $\beta$, and since $\beta(\eps z) = 0$ for all $z \in
  \BI^\eps \setminus \tilde \BI^\eps$, that last quantity vanishes as $\eps \to
  0$. 
  
  Finally,
  when $z \in \tilde \BI^\eps$,  both $\bar \mu_\alpha$ and $\bar \mu^\eps_\alpha$ are
  $p$-admissible on $Z_z^\eps$. Therefore we can apply Lemma
  \ref{lem:integralidentitybis} to obtain the
  existence of $\phi_z^\eps \in \RW^{1,p'}(Z_z^\eps)$ 
  such that
  \begin{equation}
    \label{eq:decompconv}
    \begin{aligned}
      \langle \bar \mu^\eps_\alpha - \bar \mu_\alpha,f \rangle_{\RW^{1,p}(Z_z^\eps)} &= \left( \frac{\bar
    \mu^\eps_\alpha(Z_z^\eps)}{\bar \mu_\alpha(Z_z^\eps)} - 1 \right)\langle \bar
    \mu,f\rangle_{\RW^{1,p}(Z_z^\eps)} + 
    \int_{Z_z^\eps} \nabla \phi_z^\eps
    \cdot \nabla f \de x \\
    &= 
    \left(\frac{\beta(\eps z)}{\eps^{-d} \int_{Z_z^\eps} \beta \de x}
       - 1
    \right) \langle \bar \mu_\alpha,f \rangle_{\RW^{1,p}(Q_z^\eps)}
    + 
    \int_{Z_z^\eps} \nabla \phi_z^\eps
    \cdot \nabla f \de x.
  \end{aligned}
  \end{equation}
  By uniform continuity of $\beta$ and the integral mean value theorem, since
$\bar \mu_\alpha \in \RW^{1,p}(M)^*$, the first term on the
  righthand side vanishes in the limit. 
  For the last term in \eqref{eq:decompconv}, by Hölder's inequality we have that
  \begin{equation}
    \int_{Z_z^\eps} \nabla \phi_z^\eps \cdot \nabla f \de x \le \norm{\nabla
    \phi_z^\eps}_{\RL^{p'}(Z_z^\eps)} \norm{\nabla f}_{\RL^{p}(Z_z^\eps)}
  \end{equation}
  The second part of Lemma \ref{lem:integralidentitybis} is barely too weak to
  show that the $\RL^{p'}(Z_z^\eps)$ norm of $\nabla \phi_z^\eps$ converges to zero fast
  enough. In order to prove so, we exploit the lack of scale invariance in the
  defining equation for $\phi_z^\eps$. 
  
  Define $\phi_z^{\eps,s} : s Z_z^\eps \to
  \R$ as $\phi_z^{\eps,s} := \phi_z^\eps(x/s)$. We have that $\phi_z^{\eps,s}$
  satisfies the weak differential equation
  \begin{equation}
    \forall f \in \RW^{1,p}(s Z_z^\eps), \quad \int_{s Z_z^\eps} \nabla
    \phi_z^{\eps,s} \cdot \nabla f \de x = \frac{1}{s} \int_{s \del B_z^\eps} f \, \,
    \frac{\de A}{\mu^\eps(M)} - \frac{1}{s^2}\frac{\bar \mu^\eps(Z_z^\eps)}{\bar
    \mu(Z_z^\eps)}
    \int_{s Z_z^\eps} f \de \bar \mu.
  \end{equation}
  Furthermore, $\int_{sZ_z^\eps}\phi_z^{\eps,s}d\bar\mu = 0$. Therefore, $\phi_z^{\eps,s}$ is the solution 
  of the equation~\eqref{eq:phiidentity} for measures $(s\mu^\eps(M))^{-1}\de
  A^{s\partial B_z^\eps}$ and 
  $\rd\bar\mu|_{sZ^\eps_z}$. Thus, the estimate \eqref{eq:phiestimate} in Lemma
  \ref{lem:integralidentitybis} implies that  
  \begin{equation}
    \norm{\nabla \phi_z^{\eps,s}}_{\RL^{p'}(s Q_z^\eps)} \ll_\beta(1 +
    K_{s,\eps}) s^{\frac{d-1}{p'} - 1} \frac{\eps^{\frac{\alpha}{p'}}}{1 +
    \eps^{\frac{\alpha - d}{p'}}} 
    \norm{T_{p,B(0,sR)}^{\de A^{\del B(0,sr_z^\eps)}}}.
  \end{equation}
  Here, $K_{s,\eps}$ is the Poincar\'e constant for $sQ_z^\eps$, by scaling it
  is easy to see that it remains bounded as long as $s = \bigo{\eps^{-1}}$. We
  therefore choose $s = \eps^{-1}$, and see that by
  Lemma \ref{lem:tracew11l1} we then have that
  \begin{equation}
  \norm{T_{p,B(0,sR)}^{\de A^{s\partial B(0,sr_z^\eps)}}} \ll_{d,p} 
   \eps^{\frac{\alpha - d + 1}{p}}.
  \end{equation}
  Finally, we see by scaling that
  \begin{equation}
    \begin{aligned}
    \norm{\nabla \phi_z^\eps}_{\RL^{p'}(Z_z^\eps)} &= \eps^{\frac{d}{p'} - 1}
    \norm{\nabla \phi_z^{\eps,\eps^{-1}}}_{\RL^{p'}(Z_z^\eps)} \\
    &\ll_{d,p,\beta}  \frac{\eps^{\alpha/p'}}{1
  + \eps^{\frac{\alpha - d}{p'}}}
   \eps^{\frac{p+ \alpha - d}{p}}.
  \end{aligned}
  \end{equation}
  so that 
  \begin{equation}
    \begin{aligned}
    \sum_{z \in \tilde \BI^\eps} \abs{\int_M \nabla \phi_z^\eps \cdot \nabla f
    \de x} &\le \sum_{z \in \tilde \BI^\eps} \norm{\nabla
    \phi_z^\eps}_{\RL^{p'}(Z_z^\eps)} \norm{\nabla f}_{\RL^p(Z_z^\eps)} \\
    &\ll_{d,p,\beta} \frac{\eps^{\frac{\alpha - d}{p'}}}{1 + \eps^{\frac{\alpha - d}{p'}}}
    \eps^{\frac{\alpha - d + p}{p}} \norm{\nabla f}_{\RL^p(M)} \\&= \frac{\eps^{\alpha-(d-1)}}{1 + \eps^{\frac{\alpha - d}{p'}}}\norm{\nabla f}_{\RL^p(M)},
  \end{aligned}
  \end{equation}
  where in the second step we used the inequality $\sum_{i=1}^n \abs{a_i}\leqslant n^\frac{1}{p'}\left(\sum_{i=1}^n\abs{a_i}^p\right)^{\frac{1}{p}}$.
\end{proof}

\begin{proof}[Proof of Theorem \ref{thm:homo}]
  We have already shown that conditions \ci--\cvi are satisfied. By Proposition
  \ref{prop:homneu} for the Neumann problem or \ref{prop:homstek} for the
  Steklov problem, the conditions of Proposition \ref{prop:weakenough} are
  satisfied so that we indeed have convergence of the eigenvalues and
  eigenfunctions.
\end{proof}

We finally have everything we need to prove Theorem \ref{thm:homointro}.

\begin{proof}[Proof of Theorem \ref{thm:homointro}]
  By density of continuous functions in either $\RL^{d/2}(M)$ ($d\ge 3$) or
  $\LLL(M)$ ($d=2$) there is a sequence of nonnegative
  $\beta_n$
  converging to $\beta$ in the relevant space. By Proposition \ref{L_infty_approx:prop},
  $\normallap{k}(M,g,\beta_n \rd v_g) \to \normallap{k}(M,g,\beta \rd v_g)$.
  Theorem \ref{thm:homo} along with the volume estimate \eqref{eq:volumeholes}
  entails that Theorem \ref{thm:homointro} holds for each $\beta_n$. Extracting
  a sequence from a diagonal argument yields the sequence $\Omega^\eps$ so that
  Theorem \ref{thm:homointro} holds for $\beta$.
\end{proof}

\bibliographystyle{alpha}
\bibliography{homo2}

\end{document}